\documentclass[12pt]{amsart}

\usepackage[latin1]{inputenc}					
\usepackage{amsmath,amsfonts,amssymb,amsthm}	

\usepackage{mathtools}		
\usepackage{hyperref}		
\usepackage{setspace}		
\usepackage{cite}			

\onehalfspace
\hypersetup{colorlinks=true,citecolor=blue,linkcolor=blue,urlcolor=magenta}

\numberwithin{equation}{section}

\newtheorem{theorem}{Theorem}
\newtheorem{proposition}{Proposition}
\newtheorem{lemma}{Lemma}
\newtheorem{corollary}{Corollary}

\newtheorem{definition}{Definition}

\newtheorem*{theorem*}{Theorem}
\newtheorem*{proposition*}{Proposition}
\newtheorem*{lemma*}{Lemma}
\newtheorem*{corollary*}{Corollary}

\newtheorem*{conjecture*}{Conjecture}
\newtheorem*{question*}{Question}
\newtheorem*{remark*}{Remark}
\newtheorem*{definition*}{Definition}

\renewcommand{\leq}{\leqslant}
\renewcommand{\geq}{\geqslant}

\newcommand{\Z}{\mathbb{Z}}

\newcommand{\R}{\mathbb{R}}
\newcommand{\C}{\mathbb{C}}

\newcommand{\eps}{\varepsilon}

\DeclareMathOperator{\Spec}{Spec}
\DeclareMathOperator{\Supp}{Supp}

\newcommand{\wh}{\widehat}
\newcommand{\wt}{\widetilde}

\newcommand{\E}{\mathbb{E}}

\newcommand{\U}{\mathcal{U}}
\newcommand{\ZN}{\Z/N\Z}

\begin{document}

\title{On arithmetic progressions in $A+B+C$}

\author{Kevin Henriot}

\begin{abstract}
	Our main result states that when $A,B,C$ are subsets
	of $\Z/N\Z$ of respective densities $\alpha,\beta,\gamma$, 
	the sumset $A+B+C$ contains
	an arithmetic progression of length at least
	$e^{ c(\log N)^c}$ 
	for densities
	$\alpha \geq (\log N)^{-2+\eps}$ and $\beta,\gamma \geq e^{ -c(\log N)^c }$,
	where $c$ depends on $\eps$.
	Previous results of this type required one set 
	to have density at least $(\log N)^{ -1 + o(1) }$.
	Our argument relies on the method of Croot, Laba and Sisask to establish
	a similar estimate for the sumset $A+B$ and on the recent
	advances on Roth's theorem by Sanders.
	We also obtain new estimates for the analogous 
	problem in the primes studied by Cui, Li and Xue.	
\end{abstract}

\maketitle

\section{Introduction}
\label{sec:introduction}

Let $A$ and $B$ be subsets of a cyclic group $\ZN$
of density $\alpha$ and $\beta$.
The problem of finding long arithmetic progressions
in $A+B$ has a rich history starting
with the striking result of Bourgain \cite{bourgainAPs}:
the sumset $A+B$ always contains an arithmetic progression of length
at least $e^{c(\alpha\beta\log N)^{1/3}}$
provided the densities satisfy
$\alpha\beta \geq (\log N)^{-1+o(1)}$ (and the progression is
nontrivial in this range: this will always be the case later when
we specify a range of density).
Major progress was made by Green \cite{greenAPs}
who showed that, under the same condition on densities, 
the progression could be taken as large as
$e^{c(\alpha\beta\log N)^{1/2}}$.
Sanders \cite{sandersAPs} later found a very
different proof of Green's theorem
and yet a third and relatively simple proof was provided recently
by Croot, Laba and Sisask~\cite{CLS}.

For fixed densities $\alpha$ and $\beta$,
the progression found has length $e^{c\sqrt{\log N}}$
and this has not been improved to date, while a negative result of 
Ruzsa \cite{ruzsa} says that one cannot do better than 
$e^{c(\log N)^{2/3+\eps}}$.
However when densities are allowed to decrease with $N$,
a remarkable result was obtained recently by Croot, Laba and Sisask~\cite{CLS}
Improving on a first result of Croot and Sisask \cite{CS},
they showed that the sumset $A+B$ contains
an arithmetic progression of size at least
$e^{ c (\alpha \log N)^{1/2} / (\log 2\beta^{-1})^{3/2} }$
in a range 
$\alpha (\log \tfrac{2}{\beta})^{-5} \geq C(\log N)^{-1+o(1)}$.
While the theorems of Bourgain and Green require one set
to have density at least $(\log N)^{-1/2+o(1)}$, this allows
for both sets to have density as low as $(\log N)^{-1+o(1)}$;
further, one set may even have 
exponentially small density $e^{-C(\log N)^{1/5+o(1)}}\!$.

The analogous problem for three-fold sumsets was first studied
by Freiman, Halberstam and Ruzsa \cite{FHR},
who established that the sumset $A+A+A$ contains
a much longer progression: indeed of length at least
$N^{c\alpha^3}$.
Green \cite{greenAPs} extended this to $N^{ c\alpha^{2+o(1)} }$ and
Sanders \cite{sandersAPs} to $N^{ c\alpha^{1+o(1)} }$;
however, all of these results required $\alpha \geq (\log N)^{-1/2+o(1)}$.
In contrast, the best result known for four sets or more, due
to Sanders \cite{sandersBR}, 
says that the sumset $A+A+A+A$ contains an arithmetic progression of length
$N^{c/( \log 2\alpha^{-1} )^4 }$ when $\alpha \geq e^{-C(\log N)^{1/5}}$:
in that case all the summands may be rather sparse.
In this work we investigate in detail the sumset $A+B+C$,
aiming at establishing results
valid for sparse sets $B$ and $C$
and in a large range of $\alpha$.

We now turn to the precise results,
starting with the theorem of 
Croot, Laba and Sisask \cite{CLS},
which constitutes the state-of-the-art 
on arithmetic progressions in $A+B$.

\begin{theorem}[Croot, Laba, Sisask]
	\label{thm:CLS}
	Suppose that $A$ and $B$ are subsets of $\ZN$ of 
	respective densities $\alpha$ and $\beta$.
	Then there exists an absolute constant $c > 0$ such that
	$ A + B $ contains an arithmetic progression of length at 
	least\footnote{
	We assume $N \geq 1 + \exp(e^e)$ throughout 
	to alleviate logarithmic notation.}
	\begin{align*}
		e^{ c (\alpha\log N)^{1/2} (\log 2\beta^{-1})^{-3/2} }
		\quad \text{if} \quad
		\alpha \big(\! \log \tfrac{\log N}{\beta} \big)^{-5}
		\geq (c\log N)^{-1}.
	\end{align*}
\end{theorem}

In the case of three summands,
the best bounds known are 
due to Sanders \cite{sandersAPs}.

\begin{theorem}[Sanders]
	\label{thm:sandersAPs}
	Suppose that $A,B,C$ are subsets of $\ZN$ of 
	respective densities $\alpha,\beta,\gamma$.
	Then there exists an absolute constant $c > 0$ such that
	$ A + B + C $ contains an arithmetic progression of length at least
	\begin{align*}
		N^{ c (\alpha \beta\gamma)^{1/3} }
		\quad \text{if} \quad
		(\alpha \beta\gamma)^{1/3} 
		\geq (c\log N)^{-1/2}
		(\log\log N)^{1/2}.
	\end{align*}
\end{theorem}

Cui, Li and Xue \cite{CLX} also recently 
studied the analogous problem 
for subsets of the primes.
We let $\log_k$ denote the logarithm iterated
$k$ times below.

\begin{theorem}[Cui, Li, Xue]
	\label{thm:CLX}
	Suppose that $A$ is a subset of the primes less than $N$
	of size $\alpha N / \log N$.
	Then there exist absolute positive constants 
	$c, c_0, c_1$ such that $A+A+A$ contains an
	arithmetic progression of length at least
	\begin{align*}
		\begin{array}{ll}
			N^{ c \alpha^2 / (\log 2\alpha^{-1}) }
			&\text{if} \quad
			\alpha \geq (\log_3 N)^{-c_0}, \\
			N^{ c \alpha^4 / (\log 2\alpha^{-1}) }
			&\text{if} \quad
			\alpha \geq (\log N)^{-c_1}.
		\end{array}
	\end{align*}
\end{theorem}

Their argument relies on a clever combination 
of Green's \cite{greenrestriction} 
and Helfgott and de Roton's \cite{HRrestriction} restriction theorems 
for primes with Green's \cite{greenAPs} theorem on $A+A+A$,
modified to obtain arithmetic progressions 
whose elements all have a certain number
of representations as a sum of three elements of $A$.
For lack of an existing expression, we call 
any lower bound on this number of representations
a \textit{counting lemma}, here and throughout the article.
Motivated by the application to the problem of sumsets of primes,
we set out, as a secondary objective, to provide 
counting lemmas in all our estimates; this is not essentially difficult
although it requires some care in the computations.

We now introduce our results.
We start with a simple observation which is
that the almost-periodicity results 
of Croot, Laba and Sisask \cite{CLS}
imply a version of Theorem~\ref{thm:sandersAPs}
which allows for two sets out of three to be sparse,
with density as small as $e^{-c(\log N)^{1/5}}$.

\begin{theorem}
	\label{thm:AP0}
	Suppose that $A,B,C$ are subsets of $\Z/N\Z$ of 
	respective densities $\alpha,\beta,\gamma$.
	Then there exists an absolute constant $c > 0$ such that
	$ A + B + C $ contains an arithmetic progression of length at least
	\begin{align*}
		N^{ c \alpha^2 / \log^4 (2/\alpha\beta\gamma) }
		\quad \text{if} \quad
		\alpha \big(\! \log \tfrac{2}{\alpha\beta\gamma}\big)^{-5/2} 
		\geq (c \log N)^{-1/2}
	\end{align*}
	such that each element of the progression has at least
	$\tfrac{1}{2} \alpha\beta\gamma N^2$ representations as a sum
	$x+y+z$ with $(x,y,z) \in A \times B \times C$.
\end{theorem}

While the dependency on densities $\beta$ and $\gamma$
in Theorem~\ref{thm:AP0} is satisfactory, 
the density $\alpha$ is still required to be
at least $(\log N)^{-1/2}$, and the arithmetic
progression is shorter than that of Theorem~\ref{thm:sandersAPs}
when $\alpha=\beta=\gamma$.
To overcome these limitations we turn to the argument of 
Sanders \cite{sandersAPs} to prove Theorem~\ref{thm:sandersAPs}.
The proof there is based on a density-increment strategy,
which builds on that introduced by Bourgain \cite{bourgainroth1} in the
context of Roth's theorem \cite{roth}.
Sanders' recent breakthrough \cite{sandersroth2}
in the latter problem introduced very powerful new techniques,
and these allow us to revisit the argument of \cite{sandersAPs}
so as to obtain the following.

\begin{theorem}
	\label{thm:AP1}
	Suppose that $A,B,C$ are subsets of $\Z/N\Z$ of 
	respective densities $\alpha,\beta,\gamma$.
	Then there exists an absolute constant $c > 0$ such that
	$ A + B + C $ contains an arithmetic progression of length at least
	\begin{align*}
		N^{ c \alpha / \log^5 (2/\alpha\beta\gamma) }
		\quad \text{if} \quad
		\alpha\, \big(\! \log \tfrac{2}{\alpha\beta\gamma} \big)^{-7} 
		\geq (c\log N)^{-1}
	\end{align*}
	such that each element of the progression has at least
	$e^{-(c\alpha)^{-1} \log^7 (2/\alpha\beta\gamma)}  \, N^2$ 
	representations as a sum $x+y+z$ with $(x,y,z) \in A \times B \times C$.
\end{theorem}

Note that the density of each set may now be as
low as $(\log N)^{-1+o(1)}$,
and that we may take two sets to be very sparse as before.
A result of this kind also follows from Theorem~\ref{thm:CLS}, 
since an arithmetic progression
in $A+B$ is always contained, up to translation, in $A+B+C$;
however the arithmetic progression obtained in this way 
is shorter than the one given by Theorem~\ref{thm:AP1},
unless $\gamma$ is extremely small compared with $\alpha$ and $\beta$, 
for example, when $\alpha \asymp \beta \asymp (\log N)^{-\eps}$ and 
$\gamma \asymp e^{-C(\log N)^{(1-\eps)/7}}$.
Surprisingly, the counting lemma of Theorem~\ref{thm:AP1} 
is quite a lot weaker than that of Theorem~\ref{thm:AP0}: 
this is due to the use of an iterative
argument which at each step 
places the sets $A,B,C$ in a certain Bohr set, 
whose size decreases as we iterate.

By using a generalization by Bloom~\cite{bloom} 
of the Katz-Koester transform of 
Sanders \cite{sandersroth2} to three or more sets,
we are able to go one step further in the range
of density; however, this time the
loss in the counting lemma is substantial.

\begin{theorem}
	\label{thm:AP2}
	Let $\eps \in (0,1)$ be a parameter and
	suppose that $A,B,C$ are subsets of $\ZN$ of 
	respective densities $\alpha,\beta,\gamma$.
	Then there exists an absolute constant $c > 0$ such that
	$ A + B + C $ contains an arithmetic progression of length at least
	\begin{align*}
		\exp\Big( c \alpha^{1/4} (\eps\log N)^{1/2} 
		\big(\! \log \tfrac{2}{\alpha\beta\gamma} \big)^{-7/2} \Big)
		\quad \text{if} \quad
		\alpha \big(\! \log \tfrac{2}{\alpha\beta\gamma} \big)^{-14} 
		\geq (c\eps\log N)^{-2}
	\end{align*}
	such that each element of the progression has at least
	$N^{ 2 - \eps }$ representations as a sum
	$x+y+z$ with $(x,y,z) \in A \times B \times C$.
\end{theorem}

Note that the progression obtained in this way is in fact
longer than that of Theorem~\ref{thm:AP1} in the range 
$(\log N)^{-1+o(1)} \leq \alpha \leq (\log N)^{-2/3+o(1)}$
when, say, $\alpha=\beta=\gamma$ and $\eps \asymp 1$.
Finally, we mention two applications of the above results 
to the analogous problem in the primes.
First, since Theorem~\ref{thm:AP1} comes with a counting lemma, 
its conclusion may be inserted into
the original argument of Cui, Li and Xue \cite{CLX} to derive 
two new estimates, which complement Theorem~\ref{thm:CLX}.

\begin{theorem}
	\label{thm:APprime}
	Suppose that $A$ is a subset of the primes less than $N$
	of size $\alpha N / \log N$.
	Then there exist absolute positive constants $c,c_2,c_3$ such that
	$A+A+A$ contains an arithmetic progression of length at least
	\begin{align*}
		\begin{array}{ll}
			N^{ c \alpha / (\log 2\alpha^{-1})^5 }
			&\text{if} \quad
			\alpha \geq (\log_4 N)^{-c_2}, \\
			N^{ c \alpha^2 / (\log 2\alpha^{-1})^5 }
			&\text{if} \quad
			\alpha \geq (\log_2 N)^{-c_3}.
		\end{array}
	\end{align*}
\end{theorem}

Secondly, Theorem~\ref{thm:AP2}, owing to its longer density
range, allows us to find long arithmetic progressions 
in $A+A+A$ for a dense subset $A$ of the primes 
on grounds of density alone,
that is, without appealing to restriction theorems for the primes.
This is mostly of conceptual interest, since our argument
is also quite involved, relying heavily 
on methods from \cite{sandersroth2}.
We record below the estimate that might be obtained 
from Theorem~\ref{thm:AP2}, by observing that the primes have 
asymptotic density $(\log N)^{-1}$ in the first $N$ integers
and with the usual Freiman embedding.

\begin{corollary}
	\label{thm:corprime}
	Suppose that $A$ is a subset of the primes less than $N$
	of size $\alpha N / \log N$.
	Then there exists an absolute positive constant $c$ such that
	$A+A+A$ contains an arithmetic progression of length at least
	\begin{align*}
		e^{ c (\alpha\log N)^{1/4} (\log \log N)^{-7/2} } 			
		\quad\text{if}\quad
		\alpha \geq (\log N)^{-1} (\log\log N)^{14}.
	\end{align*}
\end{corollary}

By comparison, the constant $c_1$ in Theorem~\ref{thm:CLX}
is $\tfrac{1}{45}$ in the original argument of \cite{CLX}.
The arithmetic progression given by this corollary
is, however, shorter than that of
Theorems~\ref{thm:CLX} and \ref{thm:APprime}
in the ranges prescribed there.

We make two last remarks about the shape of the above bounds.
The first is that in Theorems~\ref{thm:AP0}, \ref{thm:AP1}
and \ref{thm:AP2}, one may assume $\alpha \geq \beta \geq \gamma$
without loss of generality, and that under this assumption one may 
replace logarithmic terms $\log \tfrac{2}{\alpha\beta\gamma}$ 
by $\log \tfrac{2}{\beta\gamma}$ there. 
Secondly, we note that Theorems~\ref{thm:AP0}--\ref{thm:APprime}
and Corollary~\ref{thm:corprime} are nontrivial if and only 
if $N$ is larger than an absolute constant.

At this point we should also remark that arithmetic progressions 
may be obtained for sets much sparser than the ones considered above
by a combinatorial method of Croot, Ruzsa and Schoen \cite{sparse1},
recently generalized in \cite{sparse2},
although the results there take a rather different form.
Indeed, while the Fourier analytic methods used here typically 
find progressions of length $e^{(\log N)^c}$ in a range 
of density $\alpha \geq (\log N)^{-\delta}$,
these combinatorial methods produce shorter progressions, of 
size $(\log N)^c$, for a larger range of density 
$\alpha \geq N^{-\delta}$.

The article is now organized as follows.
Section~\ref{sec:notation} is devoted to notation
and Section~\ref{sec:bohrsets} is there to recall 
relevant facts about Bohr sets.
The proof of Theorem~\ref{thm:AP0} is given in Section~\ref{sec:CLS},
and in Section~\ref{sec:densityincr} we collect a number of facts 
on the density-increment strategy which are then used to give
the proof of Theorems~\ref{thm:AP1} and \ref{thm:AP2}
in Section~\ref{sec:APdensityincr}.
Finally, the estimates of Theorem~\ref{thm:APprime} 
and Corollary~\ref{thm:corprime} are
derived in Section~\ref{sec:APprime}, 
and comparisons with results on Roth's theorem 
are drawn in Section~\ref{sec:conclusion}.

\smallskip

\textbf{Acknowledgements.}
We should like to thank our supervisors 
Régis de la Bretèche and
Andrew Granville for discussions that
greatly helped improve the exposition
in this paper,
and we also thank Tom Sanders for many 
helpful comments.

\smallskip

\textbf{Funding.}
This work was supported by a \textit{contrat doctoral}
from Université Paris $7$.

\section{Notation}
\label{sec:notation}

Here we take a moment to introduce our notation.
It is mostly standard up to the choice of normalizations.

\textit{General setting.}
For the rest of the article we fix an integer $N \geq 2$ 
and we write $G=\ZN$.
It is clear, however, that our results are only meaningful
when densities vary with $N$ and when $N$ is large: one should
think of $N$ as such.

\textit{Functions}.
For a subset $X$ of $G$ and $x \in G$, we
define the averaging operator over $X$, and the operator of translation
by $x$ on functions $f : G \rightarrow \C$, respectively, by
\begin{align*}
	\E_{x \in X} f(x) = \frac{1}{|X|} \sum_{x \in X} f(x)  
	\quad\text{and}\quad
	\tau_x (f)(u) = f(u+x) \quad\text{for}\  u \in G.
\end{align*}
We also occasionally use the identity operator $I$ defined by $If=f$.
For any $p \geq 1$, we define the $L^p$-norm of a function $f$ on $G$ by
\begin{align*}
	\| f \|_{L^p} = \Big( \E_{x \in G} |f(x)|^p \Big)^{1/p} .
\end{align*}
We let $\| f \|_{\infty} = \sup_{x \in G} |f(x)|$ 
denote the uniform norm of $f$ over $G$.
The scalar product and the convolution 
of two functions $f, g$ are defined, respectively, by
\begin{align*}
	&&\langle f , g \rangle_{L^2} &= \E_{x \in G} f(x) \overline{g(x)} &&\\
	&\phantom{(x \in G)} &
	\text{and}\quad
	f \ast g (x) &= \E_{y \in G} f(y) g(x-y) 
	&&(x \in G).
\end{align*}
We also let $f^{(\ell)} = f \ast \dotsb \ast f$ 
denote the convolution of $f$ with itself $\ell$ times.

\textit{Fourier analysis on $\ZN$}.
We let $\wh{G}$ denote the dual group of $G$, that is,
the set of homomorphisms $\gamma : G \rightarrow \mathbb{U}$,
where $\mathbb{U}$ denotes the unit circle 
$\{ \omega \in \C : |\omega| = 1 \}$.
We define the Fourier transform $\wh{f}$ of 
a function $f : G \rightarrow \C$ by
\begin{align*}
	&\phantom{(\gamma \in \wh{G})} &
	\wh{f}(\gamma) &\coloneqq \E_{x\in G} f(x) \overline{\gamma(x)}
	&&(\gamma \in \wh{G}).
\end{align*}
The three basic formul\ae\ of Fourier analysis 
then read as follows:
\begin{align*}
	&\text{(Inversion)} &
	f(x) &= {\textstyle \sum_{\gamma \in \wh{G}} \wh{f}(\gamma) \gamma(x)}, \\
	&\text{(Parseval)} &
	{\textstyle \langle f,g \rangle_{L^2} } 
	&= {\textstyle \sum_{\gamma \in \wh{G}} f(\gamma) \overline{g(\gamma)} }, \\	
	&\text{(Convolution)} &
	\smash{ \wh{f \ast g} (\gamma) } 
	&= \textstyle{ \wh{f}(\gamma) \wh{g}(\gamma) } .
\end{align*}
For functions $g,h \colon \wh{G} \rightarrow \C$ we also write
\begin{align*}
	\| g \|_{\ell^p} 
	= \Big( \sum_{\gamma \in \wh{G}} |\wh{g}(\gamma)|^p \Big)^{1/p} 
	\quad\text{and}\quad
	\langle g, h \rangle_{\ell^2} 
	= \sum_{\gamma \in \wh{G}} g(\gamma) \overline{h(\gamma)}.
\end{align*}
Finally, for a real number $\eta > 0$
we define the $\eta$-spectrum of a function
$f\colon G \rightarrow \C$ by 
\begin{align*}
	\Spec_\eta ( f ) = 
	\{ \gamma \in \wh{G} \,:\, |\wh{f}(\gamma)| \geq \eta \| f \|_{L^1} \}.
\end{align*}

\textit{Characteristic functions and densities}.
We let $m_G$ denote 
the uniform measure on $G$ defined
by $m_G(X) = |X|/|G|$ for $X \subset G$.
More generally, when $A$ is a subset of $G$,
we let $m_A$ denote the uniform 
measure on $A$ defined by $m_A(X) = |X \cap A|/|A|$
for $X \subset G$.
We also define the normalized characteristic 
function of a subset $A$ of $G$ by 
\begin{align*}
	\mu_A = m_G(A)^{-1} \, 1_A
\end{align*} so that $\|\mu_A\|_{L^1}=1$;
note also the useful identity  $1_A \ast \mu_B (x) = m_{-B} (A-x)$.
When $B$ is a subset of $G$
we say that $A \subset B$ has relative density
$\alpha$ when $|A|=\alpha |B|$, that is,
when $m_B(A) = \alpha$.
Note the composition identity $m_G(A) = m_B(A) m_G(B)$.

\textit{Asymptotic notation.}
We let $c$ and $C$ denote absolute positive constants 
which may take different values at each occurrence.
We also make occasional use of Landau's and Vinogradov's 
asymptotic notation:
for two nonnegative functions $f$ and $g$,
we let $f = O(g)$ or $f \ll g$ indicate the fact the $f \leq C g$
for some constant $C > 0$, and $f = \Omega(g)$ or $f \gg g$ indicate
that $f \geq cg$ for some constant $c > 0$.
We write $f \asymp g$ when $f \ll g$ and $f \gg g$.

\section{Preliminaries on Bohr sets}
\label{sec:bohrsets}

Bohr sets are now a standard tool of additive combinatorics.
The definition and terminology we use
follows Sanders \cite{sandersroth2,sandersroth1}.
We also recall the fundamental properties of these sets 
which will be needed for our work.

\begin{definition}[Bohr set]
	\label{thm:bohsetdef}
	For a set of characters $\Gamma \subset \wh{G}$ and
	a real number $\delta > 0$, we let
	\begin{align*}
		B(\Gamma,\delta) = \{ x \in G \,:\, |1-\gamma(x)| 
		\leq \delta \ \ \forall \gamma \in \Gamma \}
	\end{align*}
	be the Bohr set of \textit{frequency set} $\Gamma$ and 
	\textit{radius} $\delta$.
	We define $d=|\Gamma|$ to be the \textit{dimension} of this Bohr set.
\end{definition}

Note that $|\gamma(x)|=1$ 
and therefore $|1-\gamma(x)| \leq 2$
for every $x \in G$ and $\gamma \in \wh{G}$,
so that the definition is only 
interesting for $\delta \leq 2$.
We will often denote a Bohr set simply by the letter $B$,
with associated parameters $\Gamma,\delta,d$.
There is a slight abuse of notation in doing so, as
the physical set $B$ may be the same for different frequency sets and radii:
one should formally think of $B$ as a triple $(B,\Gamma,\delta)$.
We also define the \textit{dilate} of $B$ by a factor $\rho$ by 
$B_\rho = B(\Gamma,\delta)_{\rho} \coloneqq B( \Gamma, \rho \delta )$.
Finally we say that $B'$ is a \textit{sub-Bohr set} of $B$, 
and we write $B' \leq B$,
when $\Gamma \subset \Gamma'$ and $\delta' \leq \delta$.

We now recall a standard bound on the growth of Bohr sets
which is proven in \cite[Lemma 4.20]{taovu},
albeit with a slightly different notion of Bohr set.
We indicate below the minor changes to the proof
needed to recover the following.

\begin{lemma}[Doubling ratio of Bohr sets]
	\label{thm:bohrsetdoubling}
	Suppose that $B$ is a Bohr set. Then
	\begin{align*}
		m_G(B_{1/2}) \geq 7^{-d} m_G(B).
	\end{align*}
\end{lemma}

\begin{proof}
Let $e(x) = e^{2i\pi x}$ and
write characters $\gamma : G \rightarrow S^1$
as $\gamma = e(\omega)$, where 
${\omega : G \rightarrow \R/\Z}$.
In \cite{taovu} a Bohr set of frequency set
$\Gamma$ and radius $\delta$ is defined as
\begin{align*}
	\tilde{B}(\Gamma,\delta) = \{ x : { | \omega(x) | \leq 
	\delta} \ \ {\forall \omega \in \Gamma} \},
\end{align*}
whereas here it is defined as
\begin{align*}
	B(\Gamma,\delta) = \{ x : { | 1- e(\omega(x)) | \leq 
	\delta} \ \ {\forall \omega \in \Gamma} \}.
\end{align*}
The covering argument used in the proof of \cite[Lemma~4.20]{taovu}
may be adjusted via the elementary inclusions
\begin{align*}
	\{ \omega \,:\, |1-e(\omega)| \leq 4\delta \}
	\subset \{ \omega \,:\, |\omega| \leq \delta \}
	\subset \{ \omega \,:\, |1-e(\omega)| \leq 2\pi\delta \},
\end{align*}
yielding a constant $7$ in the final bound in place of $4$ there.
\end{proof}

We record an immediate consequence of this bound.

\begin{lemma}[Growth of Bohr sets]
	\label{thm:bohrsetgrowth}
	Suppose that $B$ is a Bohr set and $\rho \in (0,1]$. 
	Then 
	\begin{align*}
		m_G(B_\rho) \geq e^{ - 6 d \log 2\rho^{-1} } m_G(B).
	\end{align*}
\end{lemma}

Observing that $B = B(\Gamma,2)_{\delta/2}$,
this in turn gives the following lemma.

\begin{lemma}[Size of Bohr sets]
	\label{thm:bohrsetsize}
	Suppose that $B$ is a Bohr set of radius
	$\delta \leq 2$.
	Then 
	\begin{align*}
		m_G(B) \geq e^{ - 6 d \log 4\delta^{-1} }.
	\end{align*}
\end{lemma}

One essential fact about Bohr sets is that they support
a lot of arithmetic structure.
A simple illustration of this principle is given by the following
easy consequence of Dirichlet's theorem on simultaneous
approximation \cite[Theorem II.1A]{schmidt}.

\begin{lemma}[Arithmetic progression in a Bohr set]
	\label{thm:APinbohrset}
	Let $B$ be a Bohr set of radius $\delta < \pi$.
	Then $B$ contains an arithmetic progression of 
	size at least $(1/2\pi)\,\delta N^{1/d}$.
\end{lemma}

We now recall the notion of regularity of Bohr sets
which is of crucial importance for the proof
of Theorems~\ref{thm:AP1} and \ref{thm:AP2}.
This is not needed for the proof of Theorem~\ref{thm:AP0},
therefore the reader only interested in that result 
may very well skip the following discussion.

Bourgain \cite{bourgainroth1} introduced the notion of 
regular Bohr sets in the context of Roth's theorem.
In that situation one often needs to work with Bohr sets
on different scales, and it is therefore desirable
that the size of dilates $B_{1+\rho}$ vary continuously
with $\rho$.

\begin{definition}[Regular Bohr set]
	\label{thm:bohrsetregdef}
	Let $C_0$ be an absolute constant.
	A Bohr set $B$ is said to be \textit{regular} for $C_0$ if
	\begin{align}
		\label{eq:bohrsetsregineq}
		&\phantom{( 0 < |\rho | < \tfrac{1}{C_0 d} )} &
		1 - C_0 |\rho| d \leq
		\frac{|B_{1 + \rho}|}{|B|}
		\leq 1 + C_0 |\rho| d
		&&\quad ( 0 < |\rho | < \tfrac{1}{C_0 d} ).
	\end{align}
\end{definition}

An essential observation of Bourgain \cite{bourgainroth1}
is that one may always ensure the regularity of a Bohr set
up to dilation by a constant factor.

\begin{lemma}[Existence of regular Bohr sets]
	\label{thm:bohrsetsexistence}
	There exists an absolute constant $C_0$ such 
	that for \textit{every} Bohr set $B$, there exists 
	$\kappa \in \big[\tfrac{1}{2},1\big)$ such that
	$B_\kappa$ is regular for $C_0$.
\end{lemma}

The proof of this result can now be found in many
places and we refer, for example, 
to Proposition~3.5 of \cite{sanders3APs}.
From now on we fix $C_0$ and we simply say 
that a Bohr set $B$ satisfying \eqref{eq:bohrsetsregineq} 
is regular.
The regularity property allows for a very useful 
averaging lemma, first formalized by 
Bourgain as \cite[Lemma~3.16]{bourgainroth1}.
The version we record below is closest
to \cite[Lemma~4.2]{greentaoinverse};
it says that Bohr sets are roughly invariant
under translation by, or averaging over, elements
of a smaller Bohr set.

\begin{lemma}[Regularity averaging lemma]
	\label{thm:bohrsetsregaveraging}
	Suppose that $B$ is a regular Bohr set and let 
	$x \in G$ and $\lambda : G \rightarrow \C$ 
	with $\|\lambda\|_{L^1}=1$.
	Then
	\begin{align*}
		&& &\| \mu_{x+B} - \mu_B \|_{L^1} \leq C_1 \rho d 
		&&\text{if}\quad x \in B_\rho, \\
		&& &\| \mu_B \ast \lambda - \mu_B \|_{L^1} \leq C_1 \rho d 
		&&\text{if}\quad \Supp(\lambda) \subset B_\rho,
	\end{align*}
	provided $\rho \leq \tfrac{1}{C_0 d}$ 
	and where $C_1 = 2 C_0$.
\end{lemma}

\begin{proof}
	Observe that 
	$\| \mu_{x+B} - \mu_B \|_{L^1} = 
	\tfrac{1}{|B|} \sum_{y\in G} |1_{x+B}(y)-1_B(y)|$
	and that $1_B$ and $1_{x+B}$ are equal 
	on $B_{1-\rho}$ and outside $B_{1+\rho}$.
	Therefore, 
	$\| \mu_{x+B} - \mu_B \|_{L^1} 
	\leq \tfrac{1}{|B|}(|B_{1+\rho}|-|B_{1-\rho}|)$
	and the first bound follows from \eqref{eq:bohrsetsregineq}.
	Summing over $x$ with weights $\lambda(x)$
	and applying the triangle inequality 
	yields the second estimate.
\end{proof}

\section{The Croot-Laba-Sisask approach}
\label{sec:CLS}

The aim of this section is to
prove Theorem~\ref{thm:AP0}. 
This result is a rather direct consequence 
of \cite[Theorem~7.1]{CLS} due to Croot, Laba
and Sisask, which says that the set of almost-periods
of a convolution is guaranteed to contain a large Bohr set.
The proof of this theorem relies on a combination of
the Croot-Sisask lemma~\cite{CS} 
and Chang's spectral lemma~\cite[Lemmas~3.1 and 3.4]{chang};
this combination was first exploited
by Sanders~\cite{sandersroth2,sandersBR}.
For our purpose we only need the following special case.

\begin{lemma}[Bohr-almost-periodicity of convolutions]
	\label{thm:CLSbohrCS}
	Let $p \geq 2$ and $\theta \in (0,1)$ be a pair of parameters.
	Suppose that $A_1 , A_2$ are subsets of $\ZN$ of 
	respective densities $\alpha_1, \alpha_2$.
	Then there exists a Bohr set $B$ such that
	\begin{align*}
		&\phantom{(x \in B)} &
		&\| 1_{A_1} \ast \mu_{A_2} - \tau_x  1_{A_1} \ast \mu_{A_2} \|_{L^p} 
		\leq \theta\alpha_1^{1/p}
		&&(x \in B)
	\end{align*}
	with dimension and radius satisfying
	\begin{align*}
		d &\leq Cp\theta^{-2} (\log \tfrac{2}{\theta\alpha_1\alpha_2})^3,  \\
		\delta &\geq c ( \theta \alpha_1 \alpha_2 / p )^C . 
	\end{align*}
\end{lemma}

\begin{proof}
Apply Theorem~7.4 of \cite{CLS} with 
$A=A_2$, $B=A_1$, and $S=G$, with 
doubling constants $K_1=2/\alpha_2$ and
$K_2 = 2/\alpha_1$, and with $\eps = \theta$.
This yields a parameter
\begin{align*}
	\delta' = c \theta \alpha_2^{1/2} \alpha_1^{1/p-1/2} \geq c\theta\alpha_2^{1/2}
\end{align*}
and a Bohr set of dimension at most
\begin{align*}
	d \leq Cp \theta^{-2} (\log 2/\delta')^2 (\log 2/\alpha_2)
	\leq Cp \theta^{-2} \big(\! \log \tfrac{2}{\theta\alpha_1\alpha_2} \big)^3
\end{align*}
and radius
\begin{align*}
	\delta = \delta' / d \geq c p^{-1} \theta^3 \alpha_2^{1/2}  
	\big(\! \log \tfrac{2}{\theta\alpha_1\alpha_2} \big)^{-3}
	\gg ( \theta \alpha_1 \alpha_2 / p )^4
\end{align*}
satisfying the desired almost-periodicity property.
The bound on $\delta$ might seem less crude once
we note that the lower bound of Lemma~\ref{thm:bohrsetsize} 
on $\log m_G(B)$ depends
linearly on $d$ and $\log 2\delta^{-1}$.
We have also been somewhat imprecise in handling 
logarithmic terms, so as not to needlessly 
clutter the main estimates:
indeed these terms have little bearing on the
quality of the final results.
\end{proof}

From Lemma~\ref{thm:CLSbohrCS} we first obtain a result slightly 
more general than Theorem~\ref{thm:AP0}
which finds a translate of a Bohr set in a sumset.
We follow the proof of the similar Theorem~1.7 
on p.~1380 of \cite{CS},
relying on little more than an elementary identity
of convolutions.

\begin{proposition}
	\label{thm:B0}
	Suppose that $A_1,A_2,A_3$ are subsets of $\Z/N\Z$ of 
	respective densities $\alpha_1,\alpha_2,\alpha_3$.
	Then there exists $z \in G$ and a Bohr set $B$ with
	\begin{align*}
		d &\leq C\alpha_1^{-2} \big(\! \log \tfrac{2}{\alpha_1\alpha_2\alpha_3} \big)^4 \\
		\delta &\geq c (\alpha_1\alpha_2\alpha_3)^C
	\end{align*}
	such that $1_{A_1} \ast 1_{A_2} \ast 1_{A_3}(y) 
	\geq \tfrac{1}{2} \alpha_1 \alpha_2 \alpha_3$
	for every $y \in z + B$.
\end{proposition}

\begin{proof}
Apply Lemma~\ref{thm:CLSbohrCS} to $A_1$ and $A_2$
with parameters $p$ and $\theta$ to be determined later.
This yields a Bohr set $B$ with dimension 
$d \leq C p\theta^{-2} \big(\! \log \tfrac{2}{\theta\alpha_1\alpha_2} \big)^3$
and radius $\delta \geq c(\theta\alpha_1\alpha_2/p)^C$
such that
\begin{align}
	\label{eq:CLSalmostp}
	&\phantom{(x \in B)} &
	&\| ( I - \tau_x ) 1_{A_1} \ast \mu_{A_2} \|_{L^p} 
	\leq \theta\alpha_1^{1/p}
	&&(x \in B).
\end{align}
Let $z \in G$ and $x \in B$ and observe that
\begin{align*}
		1_{A_1} \ast \mu_{A_2} \ast \mu_{A_3} (z)
		- 1_{A_1} \ast \mu_{A_2} \ast \mu_{A_3} (z + x)
		= \langle\, (I - \tau_x ) 1_{A_1} \ast \mu_{A_2} , 
		\tau_{-z} \mu_{-A_3} \rangle_{L^2}.
\end{align*}
Applying successively Hölder's inequality 
and \eqref{eq:CLSalmostp} we have therefore
\begin{align}
	\notag
	| 	1_{A_1} \ast \mu_{A_2} \ast \mu_{A_3} (z)
	- 1_{A_1} \ast \mu_{A_2} \ast \mu_{A_3} (z + x) | 
	&\leq \| (I-\tau_x) 1_{A_1} \ast \mu_{A_2} \|_{L^p} \|\mu_{A_3}\|_{L^q} \\
	\notag
	&\leq \theta (\alpha_1/\alpha_3)^{1/p} \\	
	\label{eq:CLSdiffcvl}
	&\leq \theta \alpha_3^{-1/p}
\end{align}
Since $\E_{z \in G} 1_{A_1} \ast \mu_{A_2} \ast \mu_{A_3} (z) = \alpha_1$,
we may pick $z$ so that 
$1_{A_1} \ast \mu_{A_2} \ast \mu_{A_3} (z) \geq \alpha_1$.
Choosing $p = 2 + \log \alpha_3^{-1}$ and 
$\theta = \alpha_1 / 2e$,
we have $\theta \alpha_3^{-1/p} \leq \alpha_1 / 2$,
and by \eqref{eq:CLSdiffcvl} we conclude that
$1_{A_1} \ast \mu_{A_2} \ast \mu_{A_3} (z+x) \geq \alpha_1/2$,
where $x \in B$ is arbitrary.
\end{proof}

We may now quickly derive Theorem~\ref{thm:AP0},
which we reproduce below with adjusted notation
for convenience.

\begin{proposition*}[Theorem~\ref{thm:AP0}]
	\label{thm:AP0restated}
	Suppose that $A_1,A_2,A_3$ are subsets of $\Z/N\Z$ of 
	respective densities $\alpha_1,\alpha_2,\alpha_3$
	and write $\wt{\alpha}=\alpha_1\alpha_2\alpha_3$.
	Then there exist absolute constants $c > 0$ and $C > 0$ such that
	$ A_1 + A_2 + A_3 $ contains an arithmetic progression of length at least
	\begin{align*}
		N^{ c \alpha_1^2 / (\log 2\wt{\alpha}^{-1})^4 }
		\quad \text{if} \quad
		\alpha_1 (\log 2\wt{\alpha}^{-1})^{-5/2} 
		\geq C(\log N)^{-1/2}
	\end{align*}
	such that each element of the progression has at least
	$\tfrac{1}{2} \wt{\alpha} N^2$ representations as a sum.
\end{proposition*}

\begin{proof}
Apply Proposition~\ref{thm:B0} to obtain a Bohr
set $B$ and an element $z \in G$ such that
$d \leq C\alpha_1^{-2} (\log 2\wt{\alpha}^{-1})^4$,
$\delta \geq c\wt{\alpha}^C$ and
$1_{A_1} \ast 1_{A_2} \ast 1_{A_3} (y) \geq \tfrac{1}{2}\wt{\alpha}$
for every $y \in z+B$.
By Lemma~\ref{thm:APinbohrset} we may find
an arithmetic progression $P \subset B$ of size
\begin{align*}
	|P| \geq \exp\bigg( 
	\frac{ c \alpha_1^2 \log N }{ (\log 2\wt{\alpha}^{-1})^4 }
	- C \log 2\wt{\alpha}^{-1} \bigg).
\end{align*}
Restricting to 
$\alpha_1^2 (\log 2\wt{\alpha}^{-1})^{-5} \geq C'(\log N)^{-1}$
with $C'$ large enough we see that 
$z+P$ is the desired arithmetic progression.
\end{proof}

\section{Preliminaries on the density-increment strategy}
\label{sec:densityincr}

The proof of Theorems~\ref{thm:AP1} and \ref{thm:AP2}
is based on the density-increment strategy used by Bourgain 
\cite{bourgainroth1,bourgainroth2} to obtain 
good bounds in Roth's theorem \cite{roth}
and later considerably expanded by Sanders 
in \cite{sandersroth2,sandersroth1}.
The base of this theory is best presented in \cite{sandersBR},
while the more advanced techniques specific to Roth's theorem 
may be found in \cite{sandersroth2,sandersroth1}.
We also use a recent refinement of those by Bloom \cite{bloom}. 
In this section we collect the main facts that we need from these references.

We first need a special case of
\cite[Lemmas~4.6 and 6.3]{sandersroth1}, which
together constitute a local version of 
Chang's spectral lemma
\cite[Lemmas~3.1 and 3.4]{chang}.

\begin{lemma}[Local spectrum annihilation]
	\label{thm:spectrumannihilation}
	Let ${ \eps \in (0,1] }$ be a parameter.
	Let $B$ be a regular Bohr set and suppose that
	$X \subset B$ has relative density $\tau$.
	Then there exists a regular Bohr set $B' \leq B$ with
	\begin{align*}
		d' \leq d+C\eps^{-2} \log 2\tau^{-1}
		\quad\text{and}\quad
		\delta' \geq c \delta / (d^2 \eps^{-2} \log 2\tau^{-1})
	\end{align*}
	such that $|1 - \gamma (x)| \leq \frac{1}{2}$ 
	for every $\gamma \in \Spec_\eps(\mu_X)$ and $x \in B'$.
\end{lemma}

\begin{proof}
Write $B=B(\Gamma,\delta)$ and
let $\Delta = \Spec_\eps (\mu_X)$.
By Sanders \cite[Lemma~4.6]{sandersroth1},
$\Delta$ has $(1,\mu_B)$-relative entropy 
$k \ll \eps^{-2}\log2\tau^{-1}$
(see \cite{sandersroth1} for the definition of this concept);
note in passing that, by the definition of entropy, $k \geq 1$.
Applying \cite[Lemma~6.3]{sandersroth1}
to $\Delta$ with $\eta=1$, we may further find a set $\Lambda$
of size at most $k$ such that,
for every $\nu\in (0,1), \rho \leq c/(dk)$, and 
$\gamma \in \Delta$,
\begin{align*}
	|1-\gamma(x)| &\ll k\nu + \rho  d^2 (k+1)
	\quad\text{uniformly in}\quad
 	x \in B( \Gamma \cup \Lambda, \min(\rho \delta, 2\nu)). 
\end{align*}
Choosing $\rho = c/(d^2 k)$ and $\nu = c/k$ with $c$
small enough we see that $|1-\gamma(x)| \leq \tfrac{1}{2}$
for $x \in B(\Gamma \cup \Lambda,c\delta/d^2 k) \eqqcolon \wt{B}$, 
and we are done upon choosing $B'=\wt{B}_\kappa$
with $\kappa \in \big[\tfrac{1}{2},1\big)$ chosen via 
Lemma~\ref{thm:bohrsetsexistence} 
such that $\wt{B}$ is regular.
\end{proof}

Note that, as in \cite{sandersAPs}, we need to keep
track of the radius of the Bohr set rather than its size,
since we are looking for arithmetic progressions
such as given by Lemma~\ref{thm:APinbohrset}.
The following is \cite[Lemma~3.8]{sandersroth2}
where we used the Bohr set given by 
Lemma~\ref{thm:spectrumannihilation} in the proof instead.
This lemma forms the backbone of the density-increment strategy.

\begin{lemma}[$L^2$ density-increment]
	\label{thm:L2increment}
	Let $\nu,\eta,\rho \in (0,1]$ be parameters.
	Let $B$ and $\dot{B} \leq B_\rho$ be regular Bohr sets. 
	Suppose that $A \subset B$ has relative density $\alpha$
	and $X \subset \dot{B}$ has relative density $\tau$.
	Write $f_A = 1_A - \alpha 1_B$, and suppose that
	$\rho \leq c\nu\alpha/d$ and
	\begin{align*}
		\sum_{\gamma\in\Spec_{\eta}(\mu_X)} | \wh{f_A}(\gamma) |^2 
		\geq \nu\alpha^2 m_G(B).
	\end{align*}
	Then there exists a regular Bohr set $\breve{B} \leq \dot{B}$ such that
	${ \| 1_A \ast \mu_{\breve{B}} \|_{\infty} \geq (1 + c\nu)\alpha }$,
	\begin{align*}
		\breve{d} 		\leq \dot{d} + C\eta^{-2}\log 2\tau^{-1}
		\quad\text{and}\quad 
		\breve{\delta} 	\geq c \dot{\delta} / ( \dot{d}^2 \eta^{-2} \log 2\tau^{-1}).
	\end{align*}
\end{lemma}

The slightly different shape of the density-increment lemma
above affects in a minor way the statement of two
results we introduce next.
The first is the Katz-Koester transform developed by 
Sanders \cite{sandersroth2};
the following is Proposition~4.1 from there.

\begin{lemma}[Katz-Koester transform]
	\label{thm:KK2sets}
	Let $\rho,\rho' \in (0,1)$ be parameters. 
	Let $B$ be a regular Bohr set, assume that
	$B'=B_{\rho}$ is regular and let $B''=B'_{\rho'}$.
	Suppose that $A \subset B$ has relative density $\alpha$ 
	and $A' \subset B'$ has relative density $\alpha'$.
	Assume that $\rho \leq c \alpha / d$ and $\rho' \leq c\alpha'/d$.
	Then either
	\begin{enumerate}
		\item	there exists a regular Bohr set $\breve{B} \leq B'$ such that
				$\| 1_A \ast \mu_{\breve{B}} \|_{\infty} \geq (1 + c) \alpha$,
				\begin{align*}
					\breve{d} \leq d + C \alpha^{-1} \log 2\alpha'^{-1} 
					\quad\text{and}\quad
					\breve{\delta} \geq c\rho(\alpha\alpha'/d)^C \delta,
				\end{align*}
		\item	or there exist $L \subset B$ with relative density $\lambda$
				and $S \subset B''$ with relative density $\sigma$, 
				such that $\lambda \gg 1$, 
				$\sigma \geq e^{-C\alpha^{-1}\log 2\alpha'^{-1}}$ and
				\begin{align*}
					1_{L} \ast 1_{S} \leq C\alpha^{-1}\, 1_A \ast 1_{A'}.
				\end{align*}
	\end{enumerate}
\end{lemma}

A second result we import is a generalization 
of the above for three of more sets due to 
Bloom \cite{bloom}; the following
is a direct consequence of the case $k=2$ of
Theorem~6.1 from there.

\begin{lemma}[Katz-Koester transform for three sets]
	\label{thm:KK3sets}
	Let $\rho,\rho' \in (0,1)$ be parameters. 
	Let $B$ be a regular Bohr set, 
	suppose that $B' = B_\rho$ is regular 
	and let $B''=B'_{\rho'}$.
	Suppose that $A \subset B$ has relative density $\alpha$
	and $A'_1,A'_2 \subset B'$ have relative 
	densities $\alpha'_1,\alpha'_2$,
	and write $\gamma=\alpha\alpha'_1\alpha'_2$.
	Assume that $\rho \leq c\alpha/d$ and $\rho' \leq c\gamma/d$.
	Then either
	\begin{enumerate}
		\item	there exists a regular Bohr set $\breve{B} \leq B'$
				such that
				$\| 1_A \ast \mu_{\breve{B}} \|_{\infty} \geq (1+c) \alpha$,
				\begin{align*}
					\breve{d} \leq d + C \alpha^{-1/2} \log 2\gamma^{-1}
					\quad\text{and}\quad
					\breve{\delta} \geq c\rho(\gamma/d)^C \delta,
				\end{align*}
		\item	or there exist $L \subset B$ with relative density $\lambda$
				and $S_1,S_2 \subset B''$ 
				with relative densities $\sigma_1,\sigma_2$ 
				such that $\lambda \gg 1$,
				$\sigma_i \geq e^{ - C \alpha^{-1/2} \log 2\gamma^{-1} }\!$,
				and
				\begin{align*}
					1_{L} \ast 1_{S_1} \ast 1_{S_2} \leq 
					C \alpha^{-2}\, 1_{A} \ast 1_{A'_1} \ast 1_{A'_2}.
				\end{align*}
	\end{enumerate}
\end{lemma}

Finally, we are going to make extensive use of the
Croot-Sisask lemma \cite{CS}, which says that
two-fold convolutions possess large sets of almost-periods.
This technique is particularly suited to prove asymmetric results
such as Theorems~\ref{thm:AP1} and \ref{thm:AP2}.
The slightly different version we quote is 
\cite[Lemma~4.3]{sandersBR} due to Sanders.

\begin{lemma}[Croot-Sisask lemma]
	\label{thm:CS}
	Let $p \geq 2$ and $\eps \in (0,1)$ be a pair of parameters.
	Let $f : G \rightarrow \C$ and $L \geq 1$ and 
	assume that $S$ and $T$ are subsets of $G$ 
	such that $|S+T| \leq L |S|$.
	Then there exist $t\in T$ and a set $X \subset T$ of
	size $|X| \geq (2L)^{-Cp/\eps^2} |T|$ 
	such that
	\begin{align*}
		\phantom{(t \in X)}&&
		\| f * \mu_{S} - \tau_y f * \mu_{S} \|_{L^p}
		&\leq \eps \| f \|_{L^p} & &(y \in X-t) .
	\end{align*}
\end{lemma}

This has the following familiar consequence, often
used implicitly throughout the literature.

\begin{lemma}[$L^p$-smoothing of convolutions]
	\label{thm:CSsmoothing}
	Let $p \geq 2$, $\ell \geq 1$, and $\theta \in (0,1)$  
	be parameters.
	Let $f : G \rightarrow \C$ and $L \geq 1$ and
	suppose that $S$ and $T$ are subsets of $G$ 
	such that $|S+T| \leq L |S|$.
	Then there exists a set $X \subset T$ of
	size $|X| \geq (2L)^{-Cp\ell^2/\theta^2} |T|$ 
	such that
	\begin{align*}
		\| f * \mu_{S} - f \ast \mu_{S} \ast \lambda_X^{(\ell)} \|_{L^p}
		\leq \theta \| f \|_{L^p}
	\end{align*}
	where $\lambda_X = \mu_X \ast \mu_{-X}$.
\end{lemma}

\begin{proof}
	Apply Lemma~\ref{thm:CS} with parameter $\eps=\theta/(2\ell)$.
	By the triangle inequality and the translation invariance of $L^p$-norms,
	we have, for every $x_1,\dots,x_\ell,x'_1,\dots,x'_\ell \in X$:
	\begin{align*}
		\| f \ast \mu_S - \tau_{x_1-x'_1+\dots+x_\ell-x'_\ell} f \ast \mu_S \|_{L^p}
		\leq \theta \| f \|_{L^p}.
	\end{align*}
	By averaging over the numerous $x_i,x'_j$ and the
	triangle inequality we recover the result.
\end{proof}

\section{Proof of Theorems~\ref{thm:AP1} and \ref{thm:AP2}}
\label{sec:APdensityincr}

We are now ready to start with the proof of our main estimates.
In this section we introduce a new piece of notation 
to make computations more bearable:
to every Bohr set $B$ we associate the 
\textit{density parameter} $b=m_G(B)$.
We start with an easy consequence of regularity that
gives us some control on the size of scaled-down sets.

\begin{lemma}[Scaling lemma]
	\label{thm:scaling}
	Let $\rho \in (0,1)$ be a parameter.	
	Let $B$ be a regular Bohr set and
	$B' \subset B_\rho$.
	Suppose that $A \subset B$ has relative density $\alpha$
	and $\rho \leq c/d$, then
	\begin{align*}
		\| 1_A \ast \mu_{B'} \|_{\infty} 
		\geq (1 - O( \tfrac{\rho d}{\alpha} ) ) \alpha.
	\end{align*}
\end{lemma}

\begin{proof}
We have, by Lemma~\ref{thm:bohrsetsregaveraging},
\begin{align*}
	\E_{x \in B} 1_A \ast \mu_{B'} (x) 
	&=\langle 1_A \ast \mu_{B'} , \mu_B \rangle_{L^2} \\
	&= \langle 1_A , \mu_B \ast \mu_{B'} \rangle_{L^2} \\
	&= \langle 1_A , \mu_B \rangle_{L^2} + 
	O\big(\, \| \mu_B - \mu_B \ast \mu_{B'} 
	\|_{L^1} \| 1_A \|_{\infty} \big) \\
	&= \alpha + O(\rho d) .
\end{align*}
Bounding the left-hand side in $\| \cdot \|_{\infty}$ norm concludes the proof.
\end{proof}

Our iterative argument initially follows
that developed by Sanders in \cite{sandersAPs},
with slight modifications to accommodate upper level sets.
We recall its principle here.
At each step, one fixes a small Bohr set $B'$
and finds a translate $A'_3$ of $A_3$ with 
relative density in $B'$ of same order as 
that of $A_3$ in $B$.
Then either $B'$ is contained in the upper level set 
$\{ { 1_{A_1} \ast 1_{A_2} \ast 1_{A'_3} } > K \}$, 
or it has nonempty intersection $\U$ with the lower level set 
${ \{ { 1_{A_1} \ast 1_{A_2} \ast 1_{A'_3} } \leq K \} }$.
The scalar product
$\langle 1_{A_1} \ast 1_{A_2} \ast 1_{A'_3} , 
1_{\U} \rangle_{L^2}$
is then unusually small for a good choice of $K$.
The usual density-increment strategy then
allows one to find a smaller Bohr set on which 
either $A_1$ or $A_2$ has increased density.
Since the density is bounded by $1$, we may iterate
this process only a finite number of times, 
after which we have found a translate of a Bohr set in
a certain upper level set.

At this point, however, we take advantage of 
two techniques from \cite{sandersroth2},
which we apply in a similar fashion.
The first is the Katz-Koester transform which 
in this situation roughly redistributes
the mass of the sets $A_1$ and $A'_3$ on 
two new sets $L$ and $S$ where $L$ is thick and $S$ is not
too small, without affecting the size of the convolution
$1_{A_1} \ast 1_{A'_3}$ excessively.
The second is the Croot-Sisask lemma
which allows one to smooth the convolution $1_L \ast 1_S$ by
a factor $\lambda_X^{(\ell)}$.
At last the density-increment strategy 
makes it possible to exploit the smallness 
of the new scalar product
${ \langle 1_L \ast 1_S \ast 1_{A_2} \ast \lambda_X^{(\ell)} , 
1_{\U} \rangle }$
to obtain a density increment on $A_2$.

Our main iterative lemma is then the following.
On a first reading the reader may wish to take
$\omega=0$ below for simplicity,
which suffices to obtain Theorem~\ref{thm:AP1}
without a counting lemma.

\begin{proposition}[Main iterative lemma]
	\label{thm:mainiteration1}
	Let $\rho,\omega \in (0,1)$ be parameters.
	Let $B$ be a regular Bohr set and 
	suppose that $B' = B_\rho$ is regular. 
	Suppose that $A_1,A_2,A_3 \subset B$
	have relative densities $\alpha_1,\alpha_2,\alpha_3$
	and write $\wt{\alpha}=\alpha_1\alpha_2\alpha_3$.
	Assume that $\rho \leq c\wt{\alpha}/d$ and
	$ \omega \leq e^{-C(d+\alpha_1^{-1}) \log (2d/\rho \wt{\alpha})}$.
	Then either
	\begin{enumerate}
		\item
		there exists a regular Bohr set $\breve{B} \leq B$ 
		such that, for some $i \in \{1,2\}$,
		\begin{align*}
			\| 1_{A_i} \ast \mu_{\breve{B}} \|_{\infty} 
			&\geq (1+c) \alpha_i, \\
			\breve{d} 
			&\leq d + C\alpha_1^{-1} (\log 2\wt{\alpha}^{-1})^4, \\
			\breve{\delta} 
			&\geq c \rho ( \wt{\alpha} / d)^C \delta,
		\end{align*}
		\item
		or there exists $x \in G$ such that
		$B' \subset \{ \, y \,:\, 1_{A_1} \ast 1_{A_2} \ast 1_{A_3} (x+y) 
		> \omega b^2 \, \}$.
	\end{enumerate}
\end{proposition}

\begin{proof}
By Lemma~\ref{thm:scaling} we may find $x \in G$ such that 
$A'_3 = (A_3 - x) \cap B'$ has relative density in $B'$ equal to
$\alpha'_3 = 1_{A_3} \ast \mu_{B'} (x) \gg \alpha_3$.
Now define
\begin{align*}
	\mathcal{U} = \{ \, y \,:\, 1_{A_1} \ast 1_{A_2} \ast 1_{A_3} (x+y) 
	\leq \omega b^2 \, \} \cap B',
\end{align*}
we may assume that $\U$ is nonempty 
since else we are in the second case of the proposition.
Note that from the inclusion $A'_3 \subset A_3 - x$ and 
the definition of $\U$, we have
\begin{align}
	\notag
	\langle 1_{A_1} \ast 1_{A_2} \ast 1_{A'_3} , \mu_{\U} \rangle_{L^2}
	&\leq \langle 1_{A_1} \ast 1_{A_2} \ast 1_{A_3-x} , \mu_{\U} \rangle_{L^2} \\
	\notag
	&= \langle 1_{A_1} \ast 1_{A_2} \ast 1_{A_3} , \mu_{x+\U} \rangle_{L^2} \\
	\label{eq:smallscalar1}
	&\leq \omega b^2
\end{align}
where $\mu_{\U}$ is well-defined 
since $\U \neq \varnothing$.
From hereon, the proof divides into three steps.

\textit{Applying the Katz-Koester transform.}
Let $\rho'=c\kappa\alpha_3/d$ and 
$B'' = B'_{\rho'}$,
where $\kappa \in \big[\tfrac{1}{2},1\big)$ is chosen via 
Lemma~\ref{thm:bohrsetsexistence} so that
$B''$ is regular.
Applying Lemma~\ref{thm:KK2sets}
to $A = A_1$ and $A' = A'_3$ with parameters
$\rho$ and $\rho'$ then
results in one of two cases.
In case (i) of that lemma we obtain 
a regular Bohr set $\breve{B} \leq B'$
such that
$\| 1_{A_1} \ast \mu_{\breve{B}}\|_{\infty} \geq (1+c) \alpha_1$,
\begin{align*}
	\breve{d} \leq d + C\alpha_1^{-1} \log 2\alpha_3^{-1}
	\quad\text{and}\quad
	\breve{\delta} \geq c\rho(\alpha_1\alpha_3/d)^C \delta,
\end{align*}
which is enough to conclude.
In case (ii), we may find $L \subset B$
with relative density $\lambda$
and $S \subset B''$ with relative density $\sigma$
such that
\begin{align}
	\label{eq:lambdasigma1}
	&\lambda \gg 1
	\quad\text{and}\quad 
	\sigma \geq e^{-C\alpha_1^{-1} \log 2\alpha_3^{-1}}, \\
	\label{eq:KKbound1}
	&1_L \ast 1_S \ll 
	\alpha_1^{-1} \, 1_{A_1} \ast 1_{A'_3}.
\end{align}
By \eqref{eq:KKbound1} we then have
\begin{align*}
	\langle 1_L \ast \mu_S , 
	1_{-A_2} \ast \mu_{\mathcal{U}} \rangle_{L^2} 
	 &= (\sigma b'')^{-1} 
	\langle 1_L \ast 1_S , 
	1_{-A_2} \ast \mu_{\mathcal{U}} \rangle_{L^2} \\
	&\ll (\alpha_1 \sigma b'')^{-1}
	\langle 1_{A_1} \ast 1_{A'_3} , 
	1_{-A_2} \ast \mu_{\mathcal{U}} \rangle_{L^2} \\
	&= (\alpha_1 \sigma b'')^{-1}
	\langle 1_{A_1} \ast 1_{A_2} \ast 1_{A'_3} , 
	\mu_{\mathcal{U}} \rangle_{L^2}.
\end{align*}
By \eqref{eq:smallscalar1} we have further
\begin{align*}
	\langle 1_L \ast \mu_S , 
	1_{-A_2} \ast \mu_{\mathcal{U}} \rangle_{L^2} 
	&\ll ( \alpha_1 \sigma b'')^{-1} \omega b^2 \\
	&= ( \lambda \alpha_1 \alpha_2 \sigma )^{-1} (b/b'') \omega
	\cdot \lambda \alpha_2 b.
\end{align*}
Recalling \eqref{eq:lambdasigma1} and applying
Lemma~\ref{thm:bohrsetgrowth} we have therefore
\begin{align*}
	\langle 1_L \ast \mu_S , 
	1_{-A_2} \ast \mu_{\mathcal{U}} \rangle_{L^2}
	\leq e^{ C (d+\alpha_1^{-1}) \log (2d/\rho \wt{\alpha}) } 
	\omega \cdot \lambda \alpha_2 b.
\end{align*}
Assuming $\omega \leq e^{-C' (d+\alpha_1^{-1}) \log (2d/\rho \wt{\alpha})}$ 
with $C'$ large enough we eventually obtain
\begin{align}
	\label{eq:Ibound1}
	\langle 1_L \ast \mu_S , 1_{-A_2} \ast \mu_{\mathcal{U}} \rangle_{L^2}
	\leq \tfrac{1}{4} \lambda \alpha_2 b.
\end{align}

\textit{Applying the Croot-Sisask lemma.}
Let $\rho'' = c\kappa' / d$ and $B''' = B''_{\rho''}$, where
${\kappa' \in \big[\tfrac{1}{2},1\big)}$ is chosen 
via Lemma~\ref{thm:bohrsetsexistence}
so that $B'''$ is regular,
and with $c$ small enough so that, by regularity of $B''$ and 
Definition~\ref{thm:bohrsetregdef},
\begin{align*}
	|S + B'''| 
	\leq |B'' + B'''| 
	\leq |B''_{1+\rho''}|
	\leq 2 |B''|
	= (2/\sigma) |S|.
\end{align*}
Applying Lemma~\ref{thm:CSsmoothing} to 
$f = 1_L$ and $T = B'''$ with 
parameters $p,\ell,\theta$ to be determined later,
we obtain a set $X \subset B'''$ of relative 
density $\tau$ with
\begin{align}
	\label{eq:Xdensity1}
	\tau \geq \exp\big( \!-C(p\ell^2/\theta^2) \log 2\sigma^{-1} \big)
	\shortintertext{such that}
	\notag
	\| 1_L \ast \mu_S - 1_L \ast \mu_S \ast \lambda_X^{(\ell)} \|_{L^p}
	\leq \theta \| 1_L \|_{L^p}.
\end{align}
By Hölder's and Young's inequalities we have therefore
\begin{align*}
	| \langle 1_L \ast \mu_S , 
	1_{-A_2} \ast \mu_{\mathcal{U}} \rangle_{L^2} &- 
	\langle 1_L \ast \mu_S \ast \lambda_X^{(\ell)} , 
	1_{-A_2} \ast \mu_{\mathcal{U}} \rangle_{L^2} | 		\\
	&\leq \| 1_L \ast \mu_S - 1_L \ast \mu_S \ast \lambda_X^{(\ell)} \|_{L^p}
	\| 1_{-A_2} \ast \mu_{\mathcal{U}} \|_{L^q} 			\\
	&\leq \theta \| 1_L \|_{L^p} \| 1_{-A_2} \|_{L^q} 		\\
	&= \theta \lambda^{1/p} \alpha_2^{1-1/p} b
\end{align*}
Choosing $p = 2 + \log \alpha_2^{-1}$ and 
$\theta = \lambda^{1-1/p} / 4e \asymp 1$, this is less 
than $\tfrac{1}{4} \lambda \alpha_2 b$,
which combined with \eqref{eq:Ibound1} shows that
\begin{align}
	\label{eq:A_2scalarbound1}
	| \langle 1_L \ast \mu_S \ast \lambda_X^{(\ell)} , 
	1_{-A_2} \ast \mu_{\mathcal{U}} \rangle_{L^2} |
	\leq \tfrac{1}{2} \lambda\alpha_2 b.
\end{align}

\textit{Obtaining an $L^2$ density increment.}
Since $\U,S,X$ are contained in $B'$, the function
$\mu_\U \ast \mu_{-S} \ast \lambda_X^{(\ell)}$ has
support in ${ (2\ell+2) B' } \subset B_{(2\ell+2)\rho}$
and we have, by Lemma~\ref{thm:bohrsetsregaveraging},
\begin{align}
	\notag
	\langle 1_L \ast \mu_S \ast \lambda_X^{(\ell)} , 
	1_B \ast \mu_{\mathcal{U}} \rangle_{L^2}
	&= \langle 1_L , 
	1_B \ast \mu_{\mathcal{U}} \ast \mu_{-S} \ast \lambda_X^{(\ell)} 
	\rangle_{L^2} \\
	\notag
	&= \langle 1_L , 1_B \rangle_{L^2} +
	O\big(\, \| 1_B - 1_B \ast \mu_{\mathcal{U}} \ast \mu_{-S} 
	\ast \lambda_X^{(\ell)} \|_{L^1} \| 1_L \|_{\infty} \big) \\
	\notag
	&= \lambda b + O( \ell \rho d b ) \\
	\label{eq:Bscalarbound1}
	&\geq \tfrac{3}{4} \lambda b
\end{align}
provided that $\rho \leq c/(\ell d)$, 
which will turn out to be the case.
Forming the balanced function 
$f_{-A_2} = 1_{-A_2}-\alpha_2 1_B$, we deduce
from \eqref{eq:A_2scalarbound1} and \eqref{eq:Bscalarbound1} that
\begin{align*}
	| \langle 1_L \ast \mu_S \ast \lambda_X^{(\ell)} , 
	f_{-A_2} \ast \mu_{\mathcal{U}} \rangle_{L^2} |
	\geq \tfrac{1}{4} \lambda\alpha_2 b.
\end{align*}
By Parseval's formula and the inequality
$\| \wh{f} \mspace{0.5mu} \|_{\infty} \leq \| f \|_{L^1}$ we have therefore
\begin{align*}
	\tfrac{1}{4} \lambda \alpha_2 b 
	&\leq \big\lvert \langle \wh{1_L} \cdot \wh{\mu_S} 
	\cdot \wh{\mu_X}^{\ell} \cdot \wh{\mu_{-X}}^{\ell} ,
	\wh{f_{-A_2}} \cdot \wh{ \mu_{\mathcal{U}} }  \rangle_{\ell^2} \big\rvert	\\
	&\leq \| \wh{\mu_S} \|_{\infty} \| \wh{ \mu_{\mathcal{U}} } \|_{\infty}
	\| \wh{1_L} \cdot \wh{f_{A_2}} \cdot \wh{\mu_X}^{2\ell} \|_{\ell^1}			\\
	&\leq \| \wh{1_L} \cdot \wh{f_{A_2}} \cdot \wh{\mu_X}^{2\ell} \|_{\ell^1}.
\end{align*}
By Cauchy-Schwarz and Parseval's identity, we then have
\begin{align*}
	\tfrac{1}{4} \lambda \alpha_2 b 
	\leq \| \wh{1_L} \|_{\ell^2} \|\wh{f_{A_2}} \cdot \wh{\mu_X}^{2\ell} \|_{\ell^2}
	= (\lambda b)^{1/2}
	\|\wh{f_{A_2}} \cdot \wh{\mu_X}^{2\ell} \|_{\ell^2}.
\end{align*}
It follows that, for some constant $c$,
\begin{align}
	\label{eq:minorarc1}
	\sum_{\gamma} |\wh{f_{A_2}}(\gamma)|^2 |\wh{\mu_X}(\gamma)|^{4\ell} 
	\geq \tfrac{1}{16} \lambda \alpha_2^2 b \geq c \alpha_2^2 b.
\end{align}
By Parseval's identity and
choosing $\ell = C \log 2\alpha_2^{-1}$ 
with $C$ large enough we have
\begin{align*}
	\smash{ \sum_{ \gamma \,:\, |\wh{\mu_{X}}(\gamma)| \leq 1/2 } }
	|\wh{f_{A_2}}(\gamma)|^2 |\wh{\mu_X}(\gamma)|^{4\ell}
	&\leq 2^{- 4 \ell } \| f_{A_2} \|_{L^2}^2 \\
	&\leq 2^{2 - 4 \ell } \alpha_2 b \\
	&\leq \tfrac{1}{2} c \alpha_2^2 b.
\end{align*}
By \eqref{eq:minorarc1} and the bound $\| \wh{\mu_X} \|_{\infty} \leq 1$, 
we have therefore
\begin{align*}
	\sum_{\gamma \in \Spec_{1/2}(\mu_X)} |\wh{f_{A_2}}(\gamma)|^2 
	\gg \alpha_2^2 b.
\end{align*}
The parameters we have chosen have size 
$p \asymp \log 2\alpha_2^{-1}$,
$\ell \asymp \log 2\alpha_2^{-1}$, 
and $\theta \asymp 1$, and therefore
by \eqref{eq:Xdensity1} and \eqref{eq:lambdasigma1},
we have
\begin{align*}
	\tau \geq \exp\big( \!-C\alpha_1^{-1} (\log 2\wt{\alpha}^{-1})^4 \big).
\end{align*} 
Since $\rho' \asymp \alpha_3/d$ and $\rho'' \asymp 1/d$, we also have
$\delta''' = c\rho (\alpha_3/d^2) \delta$.
Applying Lemma~\ref{thm:L2increment} with $A=A_2$
and for $\eta=1/2$ and some $\nu \asymp 1$ ,
we therefore obtain 
a regular Bohr set $\breve{B} \leq B'''$ such that
${ \| 1_{A_2} \ast \mu_{\breve{B}} \|_{\infty} } \geq (1+c) \alpha_2 $
and
\begin{align*}
	\breve{d} \leq d + C\alpha_1^{-1} (\log 2\wt{\alpha}^{-1})^4
	\quad\text{and}\quad
	\breve{\delta} \geq c \rho (\wt{\alpha}/d)^4 \delta,
\end{align*}
which again is enough to conclude.
\end{proof}

We are now in a position to prove the following result,
which gives slightly more structure than Theorem~\ref{thm:AP1} 
in the form of a translate of a large Bohr set.
Theorem~\ref{thm:AP1} will then follow quickly from
this proposition and Lemma~\ref{thm:APinbohrset}.

\begin{proposition}
	\label{thm:B1}
	Suppose that $A_1,A_2,A_3$ are subsets of $\ZN$ 
	of respective densities $\alpha_1,\alpha_2,\alpha_3$
	and write $\wt{\alpha}=\alpha_1 \alpha_2 \alpha_3$.
	Then there exist $z \in G$ and a Bohr set $B$ with
	\begin{align*}
		d 
		&\leq C \alpha_1^{-1} (\log 2\wt{\alpha}^{-1})^5, \\
		\delta
		&\geq \exp\big( \!-C(\log 2\wt{\alpha}^{-1})^2 \big),
	\end{align*}
	such that, for every $y \in z+B$,
	\begin{align*}
		1_{A_1} \ast 1_{A_2} \ast 1_{A_3} (y) > 
		\exp\big( \!-C\alpha_1^{-1}(\log 2\wt{\alpha}^{-1})^7 \big).
	\end{align*}
\end{proposition}

\begin{proof}
The proof proceeds by iteration of
Proposition~\ref{thm:mainiteration1}.
We construct iteratively a sequence 
of regular Bohr sets $B^{(i)}$ 
and sequences of sets 
$A_1^{(i)} \!,\, A_2^{(i)} \!,\, A_3^{(i)} \subset B^{(i)}$
of relative densities 
$\alpha_1^{(i)} \!,\, \alpha_2^{(i)} \!,\, \alpha_3^{(i)}$.
We initiate the iteration with $B^{(1)} = B(\{0\},2) = \ZN$, 
which is regular,
and with $(A_1^{(1)} \!,\, A_2^{(1)} \!,\, A_3^{(1)})=(A_1,A_2,A_3)$.
We denote by $\delta_i$, $d_i$, and $b_i$, 
respectively the radius, dimension, and
density in $G$ of $B^{(i)}$,
and we write 	
$\wt{\alpha}^{(i)}
=\alpha_1^{(i)}\alpha_2^{(i)}\alpha_3^{(i)}$.

At each step $i$, we apply 
Proposition~\ref{thm:mainiteration1} 
to the sets $A_1^{(i)},A_2^{(i)},A_3^{(i)}$
with parameters $\omega_i$ and
$\rho_i$ to be determined later.
In case (i) of that proposition 
we define $B^{(i+1)} = \breve{B}^{(i)}$, 
while in case (ii) we stop the iteration.
Whenever $B^{(i+1)}$ is defined we pick
$(x_{j,i})_{ 1 \leq j \leq 3 }$ so that,
for every $j$, $A_j^{(i+1)} \coloneqq ( A_j^{(i)} - x_{j,i} ) \cap B^{(i+1)}$
has relative density in $B^{(i+1)}$ equal to
\begin{align*}
	\alpha_j^{(i+1)}
	= 1_{ \smash{A_j^{(i)}} } \ast \mu_{ \smash{B^{(i+1)}} }(x_{j,i})
	= \| 1_{ \smash{A_j^{(i)}} } \ast \mu_{ \smash{B^{(i+1)}} } \|_{\infty}.
\end{align*}

We now assume that $B^{(i)}$ is defined
for $1 \leq i \leq n$.
Let $i < n$, our application of Proposition~\ref{thm:mainiteration1} 
then shows that there exists $j_i \in\{1,2\}$
such that $\alpha_{j_i} ^{(i+1)} \geq (1+c) \alpha_{j_i}^{(i)}$.
Choose now $\rho_i = c'\kappa_i \wt{\alpha}^{(i)}/(2 i^2 d_i)$,
where $\kappa_i \in \big[\tfrac{1}{2},1\big)$ is picked via 
Lemma~\ref{thm:bohrsetsexistence} so that $B^{(i)}_{\rho_i}$
is regular, and with $c'$ small enough so that, by
Lemma~\ref{thm:scaling},
\begin{align}
	\label{eq:densitylowerbound1}
	\alpha_j^{(i+1)}
	\geq \big( 1 - O(\rho_i d_i/\alpha_j^{(i)}) \big) \alpha_j^{(i)}
	\geq ( 1 - \tfrac{c}{2i^2} ) \alpha_j^{(i)}
\end{align}
for every $1\leq j \leq 3$. This implies that
\begin{align*}
	\alpha_1^{(i+1)}\alpha_2^{(i+1)} 
	\geq (1-c/2)(1+c) \alpha_1^{(i)}\alpha_2^{(i)}
	\geq (1+c/4) \alpha_1^{(i)}\alpha_2^{(i)},
\end{align*} 
and as a consequence the 
iteration proceeds for at most 
$n = O( \log 2\wt{\alpha}^{-1} )$ steps.
Iterating \eqref{eq:densitylowerbound1} we also obtain
\begin{align*}
	\alpha_j^{(i)} 
	\geq e^{-O(\sum_{i=1}^\infty i^{-2})} \alpha_j 
	\gg \alpha_j
\end{align*}
uniformly in $1 \leq j \leq 3$ and $1 \leq i \leq n$.
The dimension bound from Proposition~\ref{thm:mainiteration1} 
then becomes
\begin{align*}
	d_{i+1}
	\leq d_i + (C/\alpha_1^{(i)}) 
	\log^4 (2/\wt{\alpha}^{(i)})  
	\leq d_i + O\big( \alpha_1^{-1} \log^4 (2/\wt{\alpha}) \big)
\end{align*}
for $i < n$ and therefore $d_i \ll i \alpha_1^{-1} (\log 2\wt{\alpha}^{-1})^4
\ll \alpha_1^{-1} (\log 2\wt{\alpha}^{-1})^5 $ 
uniformly in $1 \leq i \leq n$.
The radius bound from Proposition~\ref{thm:mainiteration1} 
is then
\begin{align*}
	\delta_{i+1} 
	\geq ( \wt{\alpha}^{(i)} / 2 i d_i )^{O(1)} \delta_i
	\geq (\wt{\alpha}/2)^{O(1)} \delta_i
\end{align*}
for $i < n$, whence $\delta_i \geq (\wt{\alpha}/2)^{O(i)} 
\geq e^{ - O ( (\log 2\wt{\alpha}^{-1})^2 ) }$
uniformly in $1 \leq i \leq n$.

Finally, we choose $\omega_i=\omega$ 
independent of $i$ so as to satisfy 
the condition
\begin{align*}
	\omega \leq \exp\big( \!-C(d_i+(\alpha_1^{(i)})^{-1} ) 
	\log (2d_i/\rho_i\wt{\alpha}^{(i)})  \big)
\end{align*}
from Proposition~\ref{thm:mainiteration1} 
for every $1\leq i \leq n$.
From the previous dimension and radius bounds 
we see that it is enough to take 
$\omega = e^{ - C'\alpha_1^{-1} (\log 2\wt{\alpha}^{-1})^6 }$,
with $C'$ large enough.
For that choice we deduce from Lemma~\ref{thm:bohrsetsize}
and the bounds on $d_i$ and $\delta_i$ that
$\omega {b_i}^2 
\geq e^{ -O( \alpha_1^{-1} (\log 2\wt{\alpha}^{-1})^7 ) }$
uniformly in $1\leq i\leq n$.
When we are in case (ii) of Proposition~\ref{thm:mainiteration1}
we therefore find that $B_{ \rho_n }^{(n)}$ is contained in
a translate of
\begin{align*}
	\{ \, y \,:\, 1_{ \smash{A_1^{(n)}} } \ast 
	1_{ \smash{A_2^{(n)}} } \ast 1_{ \smash{A_3^{(n)}} } (y) 
	\geq \exp\big( \!- C \alpha_1^{-1} (\log 2\wt{\alpha}^{-1})^7  \big) \, \}.
\end{align*}
Since $\rho_n \geq (\wt{\alpha}/2)^{O(1)}$ and
the $A_j^{(n)}\!$ are, by construction, contained in 
translates of the $A_j$, this concludes the proof.
\end{proof}

\begin{proof}[Proof of Theorem~\ref{thm:AP1}]
Applying Proposition~\ref{thm:B1}
with $(A_1,A_2,A_3)=(A,B,C)$
and using Lemma~\ref{thm:APinbohrset} we may find an
arithmetic progression $P$ such that
\begin{align*}
	|P| \geq \exp\bigg( 
	\frac{c\alpha(\log N)}{ (\log (2/\alpha\beta\gamma))^5 } 
	- C (\log (2/\alpha\beta\gamma) )^2 \bigg)
\end{align*}
and an element $z \in G$ such that 
$ 1_{A_1} \ast 1_{A_2} \ast 1_{A_3} (y) 
\geq e^{-C\alpha^{-1} \log^7 (2/\alpha\beta\gamma) } $
for all $y \in z+P$.
Restricting to 
$\alpha (\log \tfrac{2}{\alpha\beta\gamma})^{-7} 
\geq C' (\log N)^{-1}$ with $C'$ large enough we see 
that $z+P$ is the desired arithmetic progression.
\end{proof}

We now turn to the slightly more difficult proof of
Theorem~\ref{thm:AP2}.
The main strategy is the same and we again 
start with a small scalar product
${ \langle 1_{A_1} \ast 1_{A'_3} \ast 1_{-\U} , 
1_{-A_2} \rangle }$
where $\U$ is a certain lower level set.
However, we now fully exploit the set $\U$
in applying the generalized Katz-Koester transform 
from \cite{bloom} to the 
three sets $A_1,A'_3,-\U$.
This redistributes the mass more efficiently and
accounts for the improved dependency on densities.
The rest of the proof runs similarly with
applications of the Croot-Sisask lemma and the
density-increment strategy.

This, however, requires us to assume that 
${ \U = \{ 1_{A_1} \ast 1_{A_2} \ast 1_{A'_3} \leq K \} }$
is dense enough inside a Bohr set $B'$.
We are then in a situation already 
encountered in \cite{sandersAPs} where
at each step of the iteration it either happens that 
$\U$ has low density and that the upper level set
$\U^c = { \{ 1_{A_1} \ast 1_{A_2} \ast 1_{A'_3} > K \} }$
is thick inside $B'$; 
or that a density increment can be obtained.
The following lemma makes this precise and
the reader may again let $\omega=0$ there
to obtain Theorem~\ref{thm:AP2} without a 
counting lemma.

\begin{proposition}[Main iterative lemma]
	\label{thm:mainiteration2}
	Let $\rho,v,\omega \in (0,1)$ be parameters.
	Let $B$ be a regular Bohr set and 
	assume that $B' = B_\rho$ is regular. 
	Suppose that ${A_1,A_2,A_3 \subset B}$ 
	have relative densities
	$\alpha_1,\alpha_2,\alpha_3$ and write 
	$\wt{\alpha}=\alpha_1\alpha_2\alpha_3$.
	Assume that
	$\rho \leq c\wt{\alpha}/d$
	and
	$w \leq e^{-C(d+\alpha_1^{-1/2}) \log (2d/\rho v \wt{\alpha})}$.
	Then either
	\begin{enumerate}
		\item
		there exists a regular Bohr set
		$\breve{B} \leq B$ and $i \in \{1,2\}$
		such that 
		\begin{align*}
			\| 1_{A_i} \ast \mu_{\breve{B}}\|_{\infty} 
			&\geq (1+c) \alpha_i, \\
			\breve{d} 
			&\leq d + C\alpha_1^{-1/2} (\log 2v^{-1}) (\log 2\wt{\alpha}^{-1})^4, \\
			\breve{\delta} 
			&\geq c \rho ( v \wt{\alpha} / d)^C \delta,
		\end{align*}
		\item
		or there exists $x \in G$ such that
		$\{ y \,:\, 1_{A_1} \ast 1_{A_2} \ast 1_{A_3} (x+y) > \omega b^2 \} \cap B'$
		has relative density at least $1 - v$ in $B'$.
	\end{enumerate}
\end{proposition}

\begin{proof}
The proof is in many aspects similar to that of 
Proposition~\ref{thm:mainiteration1} and therefore we
are more brief in computations.
By Lemma~\ref{thm:scaling} we may find $x \in G$
such that $A'_3 = (A_3 - x) \cap B'$ has relative density
$\alpha'_3 = 1_{A_3} \ast \mu_{B'} (x) \gg \alpha_3$ in $B'$.
Let 
\begin{align*}
	\U = \{ y \,:\, 1_{A_1} \ast 1_{A_2} \ast 1_{A_3} (x+y) \leq \omega b^2 \} \cap B'
\end{align*}
have density $u$ in $B'$;
we may assume that $u \geq v$ since else 
we are in the second case of the proposition.
Note that, by the definitions of $A'_3$ and $\U$, we have
\begin{align}
	\langle 1_{A_1} \ast 1_{A_2} \ast 1_{A'_3} , 1_{\U} \rangle_{L^2}
	\label{eq:smallscalar2}
	\leq \omega b^2 \cdot u b' \leq \omega b^2 b'.
\end{align}
From here the proof again divides into three steps.

\textit{Applying the Katz-Koester transform.}
Choose $\rho'= c v\wt{\alpha}/d$
with the help of Lemma~\ref{thm:bohrsetsexistence}
so that $B'' \coloneqq B'_{\rho'}$ is regular.
Applying Lemma~\ref{thm:KK3sets} with
$(A,A'_1,A'_2) = (A_1,-\U,A'_3)$
then results into one of two cases.
In case (i) of that lemma we obtain a regular Bohr set
$\breve{B} \leq B'$ such that
$\| 1_{A_1} \ast \mu_{\breve{B}}\|_{\infty} \geq (1+c) \alpha_1$,
\begin{align*}
	\breve{d} \leq d + C\alpha_1^{-1/2} \log (2/v\wt{\alpha})
	\quad\text{and}\quad
	\breve{\delta} \geq c\rho(v\wt{\alpha}/d)^C \delta,
\end{align*}
which is enough to conclude via the crude bound
$\log (2/v\wt{\alpha}) 
\ll (\log 2v^{-1}) (\log 2\wt{\alpha}^{-1})$.
In case (ii), we may find $L \subset B$
of relative density $\lambda$
and $S_1,S_2 \subset B''$ of relative densities 
$\sigma_1,\sigma_2$ such that
\begin{align}
	\label{eq:lambdasigma2}
	&\lambda \gg 1, \quad
	\sigma_1,\sigma_2 \geq 
	e^{ - C\alpha_1^{-1/2} \log (2/v\wt{\alpha}) }, \\
	\label{eq:KKbound2}
	&1_L \ast 1_{S_1} \ast 1_{S_2} 
	\ll \alpha_1^{-2} \, 1_{A_1} \ast 1_{-\U} \ast 1_{A'_3}.
\end{align}
In that case write
$I = \langle 1_L \ast \mu_{S_1} \ast \mu_{S_2} , 1_{-A_2} \rangle_{L^2} $
for convenience.
By \eqref{eq:KKbound2} we then have
\begin{align*}
	I 
	&\ll (\alpha_1^2 \sigma_1 \sigma_2)^{-1}  (b'')^{-2}
	\langle 1_{A_1} \ast 1_{-\U} \ast 1_{A'_3} , 1_{-A_2} \rangle_{L^2} \\
	&= (\alpha_1^2 \sigma_1 \sigma_2)^{-1}  (b'')^{-2}
	\langle 1_{A_1} \ast 1_{A_2} \ast 1_{A'_3} , 1_{\mathcal{U}} \rangle_{L^2}.
\end{align*}
By \eqref{eq:smallscalar2}, \eqref{eq:lambdasigma2}
and Lemma~\ref{thm:bohrsetgrowth} we have further
\begin{align*}
	I &\ll (\alpha_1^2 \sigma_1 \sigma_2)^{-1} (b'')^{-2} 
	\omega b^2 b' \\
	&= (\lambda \alpha_1^2 \alpha_2 \sigma_1 \sigma_2)^{-1} (b/b'') (b'/b'') \omega
	\cdot \lambda \alpha_2 b \\
	&\leq e^{ C (d+\alpha_1^{-1/2}) \log (2d/\rho v \wt{\alpha}) } \omega \cdot \lambda \alpha_2 b.
\end{align*}
Assuming 
$\omega \leq e^{-C' (d+\alpha_1^{-1/2}) \log (2d/\rho v \wt{\alpha})}$ 
with $C'$ large enough we have therefore
\begin{align}
	\label{eq:Ibound2}
	\langle 1_L \ast \mu_{S_1} , 1_{-A_2} \ast \mu_{-S_2} \rangle_{L^2} 
	= I \leq \tfrac{1}{4} \lambda \alpha_2 b.
\end{align}

\textit{Applying the Croot-Sisask lemma.}
We let $B'''=B''_{\rho''}$ with $\rho'' = c/d$
chosen such that $B'''$ is regular 
(via Lemma~\ref{thm:bohrsetsexistence})
and with $c$ small enough so that,
by the regularity of $B''$, 
$|S_1 + B'''| \leq |B''_{1+\rho''}| \leq (2/\sigma_1) |S_1|$.
Applying Lemma~\ref{thm:CSsmoothing} with $f=1_L$, $S = S_1$, $T=B'''$
and parameters $p,\ell,\theta$ to be determined later,
we obtain a set $X \subset B'''$ of relative density $\tau$ with
\begin{align}
	\label{eq:Xdensity2}
	\tau \geq \exp\big( \!-C(p\ell^2/\theta^2) \log 2\sigma_1^{-1} \big)
	\shortintertext{such that}
	\notag
	\| 1_L \ast \mu_{S_1} - 1_L \ast \mu_{S_1} \ast \lambda_X^{(\ell)} \|_{L^p}
	\leq \theta \| 1_L \|_{L^p}.
\end{align}
Proceeding exactly as in the proof of Proposition~\ref{thm:mainiteration1}
we then obtain from \eqref{eq:Ibound2} that
\begin{align}
	\label{eq:A_2scalarbound2}
	| \langle 1_L \ast \mu_{S_1} \ast \mu_{S_2} \ast \lambda_X^{(\ell)} , 1_{-A_2} \rangle_{L^2} |
	\leq \tfrac{1}{2} \lambda\alpha_2 b
\end{align}
for the choice of parameters $p = 2 + \log \alpha_2^{-1}$ and
$\theta = \lambda^{1-1/p} /4e \asymp 1$.

\textit{Obtaining an $L^2$ density increment.}
Since the support of $\mu_{S_1} \ast \mu_{S_2} \ast \lambda_X^{(\ell)}$ 
is contained in ${(2\ell+2)B' \subset B_{(2\ell+2)\rho} }$
we have, by Lemma~\ref{thm:bohrsetsregaveraging},
\begin{align}
	\label{eq:Bscalarbound2}
	\langle 1_L \ast \mu_{S_1} \ast \mu_{S_2} \ast \lambda_X^{(\ell)} , 1_B \rangle_{L^2}
	= \lambda b + O( \ell \rho d b ) \geq \tfrac{3}{4} \lambda b
\end{align}
provided that $\rho \leq c/(\ell d)$, which will turn out to be the case.
Forming the balanced function $f_{-A_2} = 1_{-A_2}-\alpha_2 1_B$, we see
from \eqref{eq:A_2scalarbound2} and \eqref{eq:Bscalarbound2} that
\begin{align*}
	| \langle 1_L \ast \mu_{S_1} \ast \mu_{S_2} \ast \lambda_X^{(\ell)} , f_{-A_2} \rangle_{L^2} |
	\geq \tfrac{1}{4} \lambda\alpha_2 b
\end{align*}

A computation entirely analogous to that in the proof of 
Proposition~\ref{thm:mainiteration1}
then shows that, choosing
$\ell = C \log 2\alpha_2^{-1}$ with $C$ large enough, 
we have
\begin{align*}
	\sum_{\gamma \in \Spec_{1/2}(\mu_X)} |\wh{f_{A_2}}(\gamma)|^2 
	\gg \alpha_2^2 b.
\end{align*}
The parameters we have chosen have size
$p \asymp \log 2\alpha_2^{-1}$, $\ell \asymp \log 2\alpha_2^{-1}$,
and $\theta \asymp 1$.
By \eqref{eq:Xdensity2}, \eqref{eq:lambdasigma2} 
and the bound 
$\log (2/v\wt{\alpha}) \ll (\log 2v^{-1})(\log 2\wt{\alpha}^{-1})$,
we have therefore
\begin{align*}
	\tau \geq \exp\big( \!-C\alpha_1^{-1/2}(\log 2v^{-1})(\log 2\wt{\alpha}^{-1})^4 \big).
\end{align*}
Since $\rho' \asymp v\wt{\alpha}/d$ and $\rho'' \asymp 1/d$,
we also have $\delta''' = c\rho (v\wt{\alpha}/d^2) \delta$.
Applying Lemma~\ref{thm:L2increment} to $A=A_2$
with $\eta = 1/2$ and some $\nu \asymp 1$,
we obtain a regular Bohr set $\breve{B} \leq B'''$ such that
\begin{align*}
 	\| 1_{A_2} \ast \mu_{\breve{B}} \|_{\infty} 
 	&\geq (1+c) \alpha_2, \\
	\breve{d} 
	&\leq d + C\alpha_1^{-1/2}(\log 2v^{-1})(\log 2\wt{\alpha}^{-1})^4, \\
	\breve{\delta} 
	&\geq c \rho (v\wt{\alpha}/d)^4 \delta,
\end{align*}
which again is enough to conclude.
\end{proof}

Owing to the shape of Proposition~\ref{thm:mainiteration2},
we now need to find arithmetic progressions
in thick subsets of Bohr sets.
This is precisely addressed by Sanders 
from \cite[Lemma~6.7]{sandersAPs}, which we now quote.

\begin{lemma}
	\label{thm:APinbohrthickset}
	Let $v \in (0,1)$ be parameter and
	let $B$ be a regular Bohr set.
	Suppose that ${ v^{-1} \leq c\delta N^{1/d} /d }$ and 
	$A \subset B$ has relative density at least $1 - v$, 
	then $A$ contains an arithmetic progression 
	of length at least $4v^{-1}$.
\end{lemma}

We now modify our iterative lemma so as to yield 
arithmetic progressions in upper-level sets and
so as to bound the number of steps in the iteration more easily.

\begin{proposition}[Final iterative lemma]
	\label{thm:finaliteration}
	Let $\rho,v,\omega \in (0,1)$ be parameters.
	Let $B$ be a regular Bohr set and 
	assume that $B'=B_{\rho}$ is regular.
	Suppose that $A_1,A_2,A_3 \subset B$ have relative densities
	$\alpha_1,\alpha_2,\alpha_3$, respectively, and write 
	$\wt{\alpha} = \alpha_1\alpha_2\alpha_3$.
	Assume that $\rho \leq c\wt{\alpha}/d$,
	\begin{align}
		\label{eq:vomega}
		v^{-1} \leq c\delta' N^{1/d} / d
		\quad\text{and}\quad
		0 \leq \omega \leq 
		\exp\big( \!-C(d+\alpha_1^{-1/2})\log (2d/\rho v \wt{\alpha}) \big).
	\end{align}
	Then either
	\begin{enumerate}
		\item
		there exists a regular Bohr set $\breve{B} \leq B'$ such that
		\begin{align*}
			{\textstyle \prod_{1\leq j \leq 3} } 
			\| 1_{A_j} \ast \mu_{\breve{B}} \|_{\infty} 
			&\geq (1+c) \wt{\alpha}, \\
			\breve{d} 
			&\leq d + C \alpha_1^{-1/2} (\log 2v^{-1})(\log 2\wt{\alpha}^{-1})^4, \\
			\breve{\delta} 
			&\geq c\rho(v\wt{\alpha}/d)^C \delta,
		\end{align*}
		\item
		or the set $\{ y \,:\, 1_{A_1} \ast 1_{A_2} \ast 1_{A_3} (y) > \omega b^2 \}$ 
		contains an arithmetic progression of length at least $4v^{-1}$.
	\end{enumerate}
\end{proposition}

\begin{proof}
	By Proposition~\ref{thm:mainiteration2} we may either find $x \in G$ 
	such that
	\begin{align*}
		\mathcal{V} = 
		\{ y \,:\, 1_{A_1} \ast 1_{A_2} \ast 1_{A_3} (y+x) > \omega b^2 \} \cap B'
	\end{align*}
	has relative density at least $1-v$ in $B'$, in which case we may conclude
	by Lemma~\ref{thm:APinbohrthickset} with $A=\mathcal{V}$;
	or we may obtain a regular Bohr set $\breve{B}$ such that
	$\|1_{A_i} \ast \mu_{\breve{B}}\|_{\infty} \geq (1+c) \alpha_i$ 
	for some $i \in \{1,2\}$ and with the 
	prescribed radius and dimension bounds.
	Picking $j,k$ such that $\{i,j,k\} = \{1,2,3\}$,
	Lemma~\ref{thm:scaling} then shows that
	\begin{align*}
		\prod_{1 \leq \ell \leq 3} \|1_{A_\ell} \ast \mu_{\breve{B}}\|_{\infty}
		\geq (1+c) (1-O(\tfrac{\rho d}{\alpha_j}))
		(1-O(\tfrac{\rho d}{\alpha_k})) \wt{\alpha}
	\end{align*}
	and assuming $\rho \leq c'\wt{\alpha}/d$ with $c'$ 
	small enough this is indeed more than $(1+c/2)\wt{\alpha}$.
\end{proof}

We are now ready for the proof of Theorem~\ref{thm:AP2},
which we quote below with adjusted notation for convenience.

\begin{proposition*}[Theorem~\ref{thm:AP2}]
	Let $\eps \in (0,1)$ be a parameter and suppose that 
	$A_1,A_2,A_3$ are subsets of $\ZN$
	of respective densities $\alpha_1,\alpha_2,\alpha_3$,
	and write $\wt{\alpha} = \alpha_1 \alpha_2 \alpha_3$.
	Then $ A_1 + A_2 + A_3 $ contains an arithmetic 
	progression $P$ of length at least
	\begin{align*}
		\exp\Big( c\eps^{1/2} \alpha_1^{1/4} (\log N)^{1/2} ( \log 2\wt{\alpha}^{-1})^{-7/2} \Big)
		\quad \text{if} \quad
		\alpha_1 (\log 2\wt{\alpha}^{-1} )^{-14} \geq C(\eps \log N)^{-2}
	\end{align*}
	and such that
	$1_{A_1} \ast 1_{A_2} \ast 1_{A_3} (x) \geq N^{ - \eps }$
	for every $x \in P$.
\end{proposition*}

\begin{proof}
The proof proceeds by iteration of
Proposition~\ref{thm:finaliteration}.
We are brief since the iteration process is very similar
to that of the proof of Proposition~\ref{thm:B1}.

We construct iteratively
a sequence of regular Bohr sets $B^{(i)}$ with
parameters $d_i,\delta_i,b_i$ and, for every 
$1 \leq j \leq 3$, a sequence of sets 
$A_j^{(i)} \subset B^{(i)}$ 
of relative density $\alpha_j^{(i)}$,
and we write 
$\wt{\alpha}^{(i)}
=\alpha_1^{(i)}\alpha_2^{(i)}\alpha_3^{(i)}$.
We initiate the iteration with $B^{(1)}=\ZN$
and $A_j^{(1)}=A_j$ for $1\leq j\leq 3$.
At each step $i$ we apply Proposition~\ref{thm:finaliteration} to
the sets $A_j^{(i)}$ with 
parameters $\rho_i$, $v$, $\omega$ to be
determined later 
(note that $v$ and $\omega$ are chosen independent of $i$),
and in case (i) we define $B^{(i+1)}=\breve{B}^{(i)}$,
while in case (ii) we stop the iteration.
For every $1 \leq j \leq 3$, we pick $x_{i,j}$ so that
$A_j^{(i+1)} \coloneqq (A_j - x_{i,j}) \cap B^{(i+1)}$
has relative density
$\alpha_j^{(i+1)} 
= \| 1_{ \smash{A_j^{(i)}} } \ast \mu_{ \smash{B^{(i+1)}} } \|_{\infty}$
in $B^{(i+1)}$, whenever $B^{(i+1)}$ is defined.

By the density increment
$\wt{\alpha}^{(i+1)} \geq (1+c) \wt{\alpha}^{(i)}$
from Proposition~\ref{thm:finaliteration} we
see that the iteration stops after
at most $n = O( \log 2\wt{\alpha}^{-1} )$ steps.
We choose 
$\rho_i = c\wt{\alpha}^{(i)}/(i^2 d_i)$ such that 
$B^{(i)}_{\rho_i}$ is regular 
(via Lemma~\ref{thm:bohrsetsexistence}).
By Lemma~\ref{thm:scaling} we then have
$\alpha_j^{(i+1)} \geq ( 1 - O(i^{-2}) ) \alpha_j^{(i)}$ 
for every $i,j$ and therefore
$\alpha_j^{(i)} \geq e^{-O(\sum_{i=1}^{\infty} i^{-2}) } \alpha_j
\gg \alpha_j$
uniformly in $1 \leq j \leq 3$ and $1 \leq i \leq n$.
We then have, from the bounds of Proposition~\ref{thm:finaliteration},
\begin{align*}
	d_{i+1} &\leq d_i + C\alpha_1^{-1/2} 
	(\log 2v^{-1})(\log 2\wt{\alpha}^{-1})^4
\end{align*}
for $i < n$ and therefore 
$d_i \leq C\alpha_1^{-1/2} (\log 2v^{-1}) (\log 2\wt{\alpha}^{-1})^5$
uniformly in $1 \leq i \leq n$.
Bounding crudely 
$\log (2/v \wt{\alpha}) \ll (\log 2v^{-1}) (\log 2\wt{\alpha}^{-1})$,
we also have
\begin{align*}
	\delta_{i+1} &\geq 
	\exp\big( \!-C(\log 2v^{-1})(\log 2\wt{\alpha}^{-1}) \big) \, \delta_i
\end{align*}	
for $i < n$ and therefore
$\delta_i \geq 
\exp\big( \!-C(\log 2v^{-1})(\log 2\wt{\alpha}^{-1})^2 \big)$
uniformly in $1 \leq i \leq n$.

We now choose $v$ and $\omega$ 
so that \eqref{eq:vomega} is satisfied at every step.
From the previous dimension and radius bounds, 
we see that a sufficient condition for $v$ is
\begin{align*}
	\log 2v^{-1} \leq
	\frac{ c\alpha_1^{1/2} \log N }{ (\log 2v^{-1})(\log 2\wt{\alpha}^{-1})^5 }
	- C (\log 2v^{-1})(\log 2\wt{\alpha}^{-1})^2 .
\end{align*}
We choose $v$ defined by 
$\log 2v^{-1} = c' \eps^{1/2} \alpha_1^{1/4}(\log N)^{1/2} (\log 2\wt{\alpha}^{-1})^{-7/2}$
with $c'$ small enough so as to satisfy this;
since $\log 2v^{-1} \in [\, \log 2, +\infty)$,
this requires
$\alpha_1 (\log 2\wt{\alpha}^{-1})^{-14} \geq C (\eps \log N)^{-2}$
for a certain large enough $C$.
Bounding again crudely
$\log (2/v\wt{\alpha}) \ll (\log 2v^{-1})(\log 2\wt{\alpha}^{-1})$,
we also see that a sufficient condition for $\omega$ 
to satisfy \eqref{eq:vomega} is
\begin{align*}
	\omega \leq 
	\exp\big( \!-C\alpha_1^{-1/2} (\log 2v^{-1})^2 (\log 2\wt{\alpha}^{-1})^6 \big)
\end{align*}
which allows for the choice 
$\omega = N^{ - (c\eps /\log 2\wt{\alpha}^{-1}) }$
upon inserting the above expression of $\log 2v^{-1}$.
From Lemma~\ref{thm:bohrsetsize} and the choices
of $v$ and $\omega$, we eventually obtain 
$\omega b_i^2 \geq N^{ - \eps }$
uniformly in $1 \leq i \leq n$. 
When we are in case (ii) of Proposition~\ref{thm:finaliteration},
we have therefore found the desired arithmetic progression.
\end{proof}

\section{Arithmetic progressions in sumsets of sets of primes}
\label{sec:APprime}

We now consider applications of 
Theorems~\ref{thm:AP1}~and~\ref{thm:AP2}
to the problem of finding arithmetic progressions 
in $A+A+A$, for $A$ a subset of the primes.
This problem was first considered by
Cui, Li and Xue in \cite{CLX}.
In that paper a connection with the original problem on 
arithmetic progressions in sumsets 
of sets of integers was outlined
and exploited via the original theorem of Green on $A+A+A$,
which finds an arithmetic progression of size $N^{c\alpha^2}$
in this sumset when $A$ has density $\alpha$.
To obtain Theorem~\ref{thm:APprime} we exploit the same 
connection, taking advantage of the slightly longer progression
given by Theorem~\ref{thm:AP1}.
Corollary~\ref{thm:corprime} is obtained differently,
by a direct application of Theorem~\ref{thm:AP2}.

We denote by $\log_k$ the logarithm iterated $k$ times
and we let $n$ be a large enough integer.
We also recall that when $G,H$ are two groups,
a Freiman $3$-isomorphism from $A \subset G$ 
to $B \subset H$
is a map $\phi \colon A \rightarrow B$ 
such that, for every $(a_i)_{1\leq i\leq 3}$ and
$(a'_i)_{1\leq i\leq 3}$ in $A^3$, 
$\sum_i a_i = \sum_i a'_i$
if and only if
$\sum_i \phi(a_i) = \sum_i \phi(a'_i)$;
we refer the reader to \cite[Section~5.3]{taovu}
for the properties of such maps. 
The following can be extracted from the computations
of~\cite{CLX}.

\begin{proposition}
	\label{thm:CLXcomputation1}
	Let $\eps, \delta \in (0,1)$
	and suppose that $A$ has density 
	$\alpha$ in $\{1,\dots,n\} \cap \mathcal{P}$.
	Then there exist an integer $N$ such that  $n / (\log n) \ll N \ll n$,
	a subset $A'$ of $A$ which is Freiman 3-isomorphic
	to a subset $A''$ of $\ZN$,
	a function $f$ on $\ZN$ with support in $A''$, 
	and a subset $A_1$ of $\ZN$ of density at least $c\alpha$
	such that 
	\begin{align}
		\label{eq:triplefcvl}
		&&\phantom{(x \in G)}
		f \ast f \ast f (x) 
		&\geq \alpha^3 \, 1_{A_1} \ast 1_{A_1} \ast 1_{A_1} (x)
		- O( \eps + \delta^{1/2} )
		&&(x \in G)
	\end{align}
	provided $C (\log_4 N)/(\log_2 N) \leq (\eps/2\pi)^{ C\delta^{-5/2} }$.
\end{proposition}

\begin{proposition}
	\label{thm:CLXcomputation2}
	Let $\eps, \delta \in (0,1)$
	and suppose that $A$ has density 
	$\alpha$ in $\{1,\dots,n\} \cap \mathcal{P}$.
	Then there exist an integer $N$ such that $n^{1/2} \ll N \ll n$,
	a subset $A'$ of $A$ which is Freiman 3-isomorphic
	to a subset $A''$ of $\ZN$,
	a function $g$ on $\ZN$ with support in $A''$
	and a subset $A_1$ of $\ZN$ of density at least $c\alpha^2$
	such that
	\begin{align*}
		&&\phantom{(x \in G)}
		g \ast g \ast g (x) 
		&\geq \alpha^3\, 1_{A_1} \ast 1_{A_1} \ast 1_{A_1} (x)
		- O( \eps + \delta^{1/2} )
		&&(x \in G)
	\end{align*}
	provided $\delta^{-5/2} \log 2\eps^{-1} \leq c \log N$.
\end{proposition}

\begin{proof}[Proof of Theorem~\ref{thm:APprime}]
To obtain the first estimate we
apply Proposition~\ref{thm:CLXcomputation1}.
Since $A_1$ has density at least
$c \alpha$, we know by Theorem~\ref{thm:AP1} 
that $A_1+A_1+A_1$ contains an arithmetic progression $P$
of length at least $N^{ c\alpha/(\log 2\alpha^{-1})^5 }$ 
such that, for every $x \in P$,
\begin{align*}
	1_{A_1} \ast 1_{A_1} \ast 1_{A_1} (x) 
	\geq \exp\big( \!-C\alpha^{-1} (\log 2\alpha^{-1})^7 \big).
\end{align*}
Choosing $\eps = \delta = \exp( -C'\alpha^{-1} (\log 2\alpha^{-1})^7 )$
with $C'$ large enough it then follows 
from \eqref{eq:triplefcvl} that
$f \ast f \ast f (x) > 0$ for all $x \in P$,
and therefore that $P \subset A'' + A'' + A''$.
Pulling back to $A' \subset A$ by the Freiman
isomorphism we are done provided 
$\delta^{-5/2} \log 2\eps^{-1} \leq c \log_3 N$,
which is satisfied for $\alpha \geq C(\log_5 N)^7 / \log_4 N$.

To obtain the second estimate we apply
Proposition~\ref{thm:CLXcomputation2}, 
where this time $A_1$ has density at least $c\alpha^2$.
Theorem~\ref{thm:AP1} then yields a progression 
$P \subset A_1+A_1+A_1$  of length at least 
$N^{c\alpha^2 / (\log 2\alpha^{-1})^5}$ such that
\begin{align*}
	1_{A_1} \ast 1_{A_1} \ast 1_{A_1} (x) \geq 
	\exp\big( \!-C\alpha^{-2} (\log 2\alpha^{-1})^7 \big),
\end{align*}
and choosing
$\delta=\eps=\exp(-C'\alpha^{-2}(\log 2\alpha^{-1})^7 )$
we may conclude as before provided 
\begin{align*}
	e^{C\alpha^{-2}(\log 2\alpha^{-1})^7 } {\leq c\log N}.
\end{align*}
This is certainly satisfied for 
$\alpha \geq C (\log_3 N)^{7/2} / (\log_2 N)^{1/2}$.
\end{proof}

\begin{proof}[Proof of Corollary~\ref{thm:corprime}]
	The projection $\pi \colon \Z \rightarrow \Z/6N\Z$ is a 
	Freiman $3$-isomorphism from $A \subset \{1,\dots,N\}$ 
	to $A' \coloneqq \pi(A)$ which
	preserves arithmetic progressions.
	Note that $A'$ has density $\gg \alpha / \log N$
	in $\Z/6N\Z$.
	Applying Theorem~\ref{thm:AP2} with $A=B=C=A'$,
	$\eps=\tfrac{1}{2}$ and pulling
	back to $\Z$ then concludes the proof.
\end{proof}

\section{Remarks and conclusion}
\label{sec:conclusion}

There is a strong parallel between the 
quantitative results one can obtain 
about arithmetic progressions in sumsets
and on Roth's theorem
by the density-increment strategy of \cite{sandersroth2}.
Indeed the limitation in the range of density 
in both problems is similar.
To see this, consider a subset $A$ of $\ZN$ of density $\alpha$.
Sanders \cite{sandersroth2} then showed that when 
$\alpha \geq (\log N)^{-1+o(1)}$,
there exists a nontrivial three-term 
arithmetic progression in $A$,
which Bloom \cite{bloom} 
generalized to show (in particular) that for
$\alpha \geq (\log N)^{-2+o(1)}$, 
any translation-invariant
equation in four variables 
has a nontrivial solution in $A$.
By comparison, the same density-increment strategy
applied to our problem
can be made to obtain a long progression in $A+A$
in the range $\alpha \geq (\log N)^{-1+o(1)}$
(although this is not pursued here, since the argument of 
\cite{CLS} is simpler in this case) 
and, by Theorem~\ref{thm:AP2}, it yields one in $A+A+A$ for 
$\alpha \geq (\log N)^{-2+o(1)}$.
It is therefore likely that any improvement of
this technique would result in a better density dependency
in both problems.

\bibliographystyle{amsplain}
\bibliography{apsumsets_arxiv2}

\bigskip

\textsc{\footnotesize Département de mathématiques et de statistique,
Université de Montréal,
CP 6128 succ. Centre-Ville,
Montréal QC H3C 3J7, Canada
}

\textit{\small Email address: }\texttt{\small henriot@dms.umontreal.ca}

\end{document}